\documentclass[a4paper,10pt]{amsart}

\usepackage{amssymb,epsfig,amsmath,amsthm}
\usepackage{stmaryrd}
\usepackage{enumerate}
\usepackage{cite}
\usepackage[all,cmtip]{xy}
\usepackage{color}
\usepackage{url}
\usepackage{verbatim} 
\usepackage{soul}

\DeclareRobustCommand{\SkipTocEntry}[4]{}

\newtheorem{lemma}{Lemma}[section]
\newtheorem{prop}[lemma]{Proposition}
\newtheorem{thm}[lemma]{Theorem}
\newtheorem{thm'}[lemma]{``Theorem''}
\newtheorem*{thm*}{Theorem}

\newtheorem{cor}[lemma]{Corollary}

\newtheorem*{propwo}{Proposition 4.1}

\theoremstyle{definition}
\newtheorem{defn}[lemma]{Definition}

\newtheorem{example}[lemma]{Example}

\theoremstyle{remark}
\newtheorem{rmk}[lemma]{Remark}

\newtheorem{note}[lemma]{Notation}

\newtheorem{warning}[lemma]{Warning}

\newcommand{\cC}{\mathcal C}
\newcommand{\cD}{\mathcal D}
\newcommand{\cE}{\mathcal E}

\newcommand{\cK}{\mathcal K}

\newcommand{\cP}{\mathcal P}

\newcommand{\F}{\mathbb F}
\newcommand{\M}{\mathbb M}
\newcommand{\Q}{\mathbb Q}

\newcommand{\W}{\mathbb W}
\newcommand{\Z}{\mathbb Z}

\newcommand{\fm}{\mathfrak m}

\newcommand{\E}{E_{\infty}}

\newcommand{\nc}{\color{black}}

\newcommand{\Hom}{\mathrm{Hom}}
\newcommand{\Ch}{\mathrm{Ch}(\mathrm{Mod}_R)}
\newcommand{\Chhat}{\mathrm{Ch}(\widehat{\mathrm{Mod}}_R)}

\makeatletter
\ifx\SK@label\undefined\let\SK@label\label\fi
 \let\your@thm\@thm
 \def\@thm#1#2#3{\gdef\currthmtype{#3}\your@thm{#1}{#2}{#3}}
 \def\mylabel#1{{\let\your@currentlabel\@currentlabel\def\@currentlabel
  {\currthmtype~\your@currentlabel}
 \SK@label{#1@}}\label{#1}}
 \def\myref#1{\ref{#1@}}
\makeatother

\newcommand{\Mod}{\mathrm{Mod}}
\newcommand{\Modhat}{\widehat{\mathrm{Mod}}}
\newcommand{\Modff}{\mathrm{Mod}^{\mathrm{ff}}}
\newcommand{\Modflat}{\mathrm{Mod}^{\flat}}

\newcommand{\Comp}{\mathrm{Mod}^{\mathrm{perf}}}
\newcommand{\LComp}{\mathrm{LMod}^{\mathrm{perf}}}
\newcommand{\RComp}{\mathrm{RMod}^{\mathrm{perf}}}

\newcommand{\Top}{\mathrm{Top}}
\newcommand{\Sp}{\mathrm{Sp}}

\newcommand{\MM}{\hat{\mathbb{M}}}
\newcommand{\T}{\mathbb{T}}
\newcommand{\That}{\hat{\mathbb{T}}}
\newcommand{\D}{\mathbb{P}}
\newcommand{\LD}{L_K\mathbb{P}}
\newcommand{\rep}{\mathrm{Rep}}
\newcommand{\Monad}{\mathrm{Monads}}
\newcommand{\Lie}{\mathbb{L}}
\newcommand{\Hinf}{\mathbb{H}}
\newcommand{\HH}{\mathbb{H}_{\infty}}
\newcommand{\Torhat}{\widehat{\mathrm{Tor}}}

\makeatletter
\newcommand*{\doublerightarrow}[2]{\mathrel{
  \settowidth{\@tempdima}{$\scriptstyle#1$}
  \settowidth{\@tempdimb}{$\scriptstyle#2$}
  \ifdim\@tempdimb>\@tempdima \@tempdima=\@tempdimb\fi
  \mathop{\vcenter{
    \offinterlineskip\ialign{\hbox to\dimexpr\@tempdima+1em{##}\cr
    \rightarrowfill\cr\noalign{\kern.5ex}
    \rightarrowfill\cr}}}\limits^{\!#1}_{\!#2}}}
\newcommand*{\triplerightarrow}[1]{\mathrel{
  \settowidth{\@tempdima}{$\scriptstyle#1$}
  \mathop{\vcenter{
    \offinterlineskip\ialign{\hbox to\dimexpr\@tempdima+1em{##}\cr
    \rightarrowfill\cr\noalign{\kern.5ex}
    \rightarrowfill\cr\noalign{\kern.5ex}
    \rightarrowfill\cr}}}\limits^{\!#1}}}
\makeatother

\newcommand{\lcolim}{\widehat{\mathrm{colim}}}
\newcommand{\colim}{\mathrm{colim}}

\newcommand{\Def}{\textbf}

\newcommand{\Alg}{\mathrm{Alg}}
\newcommand{\Alghat}{\widehat{\mathrm{Alg}}}
\newcommand{\Com}{\mathrm{Alg}}
\newcommand{\Comhat}{\widehat{\mathrm{Alg}}}
\newcommand{\Set}{\mathrm{Set}}

\newcommand{\ral}{\xrightarrow} 
\newcommand{\surj}{\twoheadrightarrow}

\newcommand{\PP}{\hat{\D}}

\newcommand{\al}{\alpha}
\newcommand{\io}{\iota}
\newcommand{\la}{\lambda}
\newcommand{\La}{\Lambda}
\newcommand{\phy}{\varphi}
\newcommand{\te}{\theta}
\newcommand{\Si}{\Sigma}

\newcommand{\abs}[1]{\lvert #1 \rvert}

\newcommand{\lan}{\left\langle}
\newcommand{\llb}{\llbracket}
\newcommand{\op}{\oplus}
\newcommand{\ot}{\otimes}
\newcommand{\ran}{\right\rangle}
\newcommand{\rrb}{\rrbracket}
\newcommand{\sm}{\wedge}

\newcommand{\id}{\mathrm{id}}

\DeclareMathOperator{\coker}{coker}
\DeclareMathOperator{\Ext}{Ext}
\DeclareMathOperator{\Fun}{Fun}
\DeclareMathOperator{\holim}{lim}
\DeclareMathOperator{\Lan}{Lan}
\DeclareMathOperator{\TAQ}{TAQ}

\newcommand{\End}{\mathrm{End}}
\newcommand{\Ab}{\mathrm{Ab}}

\title{Completed power operations for Morava $E$-theory}
\author{Tobias Barthel and Martin Frankland}
\date{\today}

\begin{document}

\begin{abstract}
We construct and study an algebraic theory which closely approximates the theory of power operations for Morava $E$-theory, extending previous work of Charles Rezk in a way that takes completions into account. These algebraic structures are made explicit in the case of $K$-theory. Methodologically, we emphasize the utility of flat modules in this context, and prove a general version of Lazard's flatness criterion for module spectra over associative ring spectra.
\end{abstract}

\maketitle

\tableofcontents

\section{Introduction}\label{intro}

If $E$ is a commutative ring spectrum, then the homotopy groups of any commutative $E$-algebra $A$ carry important extra structure known as power operations. Historically, one of the first instances of power operations are the Dyer-Lashof operations, acting on the mod $p$ homology of infinite loop spaces, and more generally on the homotopy of any commutative $H\F_p$-algebra $A$. The structure present on $\pi_*A$ is captured by a monad on graded $\F_p$-vector spaces, which can be described explicitly \cite[\S IX.2]{hinfty}. This provides an algebraic theory of power operations for $H\F_p$. 

A central object of study in the chromatic approach to stable homotopy theory is Morava $E$-theory $E=E_h$, which is a commutative ring spectrum associated with a universal deformation of a height $h$ formal group law over the field $\F_{p^h}$. Since Ando's thesis \cite{andothesis}, power operations for Morava $E$-theory have found various applications in chromatic homotopy theory and it is an important problem to better understand this structure. Generalizing the work of McClure and Bousfield for $p$-complete $K$-theory $E_1=K^{\wedge}_p$ , Rezk \cite{congruence} constructs for any height $h$ a monad $\T: \Mod_{E_*} \to \Mod_{E_*}$ which approximates the structure of power operations on the homotopy of commutative $E_h$-algebras. 

Using this functor, the ring of (additive) power operations for Morava $E$-theory is defined as 
\[ \Gamma = \End(U\colon \Alg_{\T} \to \Ab) \cong \bigoplus_{k \ge 0} \coker\left ( \bigoplus_{0<j<p^k} E^{\wedge}_0 B(\Sigma_j \times\Sigma_{p^k-j}) \to E^{\wedge}_0 B\Sigma_{p^k}\right )\]
where $U$ denotes the degree $0$ part of the forgetful functor, $E^{\wedge}$ is completed Morava $E$-homology, and the maps are induced by the inclusions $\Sigma_j \times\Sigma_{p^k-j} \to \Sigma_{p^k}$.  Rezk then proves that $\Gamma$ is Koszul in an appropriate sense, which allows for the construction of small algebraic resolutions of $\Gamma$-modules.  

However, the algebraic approximation functor $\T$ ignores an important piece of structure present on the homotopy groups of commutative $E$-algebras:  In order to get a well-behaved theory of algebras, one has to work with $K(h)$-local algebras throughout, as we shall explain in more detail in Section \ref{prelim1}. The homotopy groups of a $K(h)$-local commutative $E$-algebra then naturally take values not merely in $\Mod_{E_*}$, but rather in the subcategory $\Modhat_{E_*}$ of $L$-complete $E_*$-modules. 

Motivated by this observation, our first goal in this paper is to construct, in a universal way, an improved version of the algebraic approximation functor   
\[\That\colon \Modhat_{E_*} \to \Modhat_{E_*}\]
that takes completions into account. It comes with a comparison map 
\[\alpha: \That\pi_*L_{K} M \to \pi_*L_{K}\D M \]
for any $E$-module $M$, which is an isomorphism if $M$ is flat as an $E$-module. This should be contrasted with the fact that the uncompleted comparison map is an isomorphism when $M$ is finitely generated and free, but rarely otherwise.

As our main theorem (\myref{mainthm}), we prove that the completed algebraic approximation functor gives rise to an algebraic theory of completed power operations. 

\begin{thm*}
The completed algebraic approximation functor $\That \colon \Modhat_{E_*} \to \Modhat_{E_*}$ admits a natural monad structure, compatible with the one on $L_K\D$.
\end{thm*}

This result is important for our understanding of the theory of power operations for Morava $E$-theory. Some of the potential applications are discussed in Section \ref{appli}.

Interestingly, the proof of monadicity of the uncompleted functor $\T$ does not carry over to this setting, due to some peculiar properties of $\Modhat_{E_*}$. Since the proof relies heavily on the algebraic and homological structure of this category, we survey the theory of $L$-complete modules in an appendix, see Appendix \ref{appendix1}. This section also contains a number of new results which are of independent interest, and our hope is to convince the reader that $\Modhat_{E_*}$ is a very natural category appearing in algebraic topology. One technical novelty of our approach is the emphasis on flat modules in the study of $L$-complete modules, for which we provide several equivalent characterizations. On the topological side, we prove a general version of Lazard's theorem for flat module spectra over associative ring spectra, generalizing previous work by Lurie and Rezk.

Our second main objective is to make the structure on the completed algebraic approximation functor explicit in the height $1$ case, i.e., for $p$-complete $K$-theory. At height $1$, the explicit structure of the power operations is understood, by earlier work of McClure \cite[\S IX]{hinfty} and Bousfield \cite{bousfield}. In contrast, few explicit formulas are known at higher heights. The height $2$ case has been calculated by Rezk at the prime $2$ \cite{ht2p2} and Zhu at the prime $3$ \cite{ht2p3}; this seems to be the state of the art as of this writing. In Section \ref{height1}, we present an alternate, formula-based proof of the height $1$ case, resting on the identification of the algebraic approximation functors in Section \ref{opsonK}.

\begin{thm*}
At height $h=1$, the monad $\T \colon \Mod_{E_*} \to \Mod_{E_*}$ is the free $\Z/2$-graded $\te$-ring over the ground $\te$-ring $\Z_p$.
\end{thm*}

We then illustrate how one can take completions into account when doing computations with the functors $\T_n$, and in particular how the non-linearity of $\T_n$ comes into play. The point of view in Section \ref{height1} is complemented by a representation theoretic interpretation in Appendix \ref{appendix2}.

\subsection*{Organization}

We start in Section \ref{prelim} with a brief recollection of free algebras over a commutative ring spectrum in general, and then explain what is special in the case of $K(h)$-local commutative $E$-algebras, where $E$ denotes Morava $E$-theory. After that, we prove a general version of Lazard's characterization of flat module spectra over arbitrary $\mathbb{E}_1$-ring spectra. In Section \ref{aaf}, we introduce Rezk's algebraic approximation functors and construct our completed analogues, studying some of their properties along the way.  We formulate our main theorem, the proof of which is the subject of Section \ref{proofofthm}, and mention some of the future applications.

Section \ref{compstructures} serves as an intermediate discussion of the compatibility between $L$-completions and certain algebraic structures. In particular, we recall facts about $\la$- and $\te$-rings, which play an important role in the following two sections. Sections \ref{opsonK} and \ref{height1} contain the alternative proof of the main theorem at height $1$, based on explicit formulas.

There are two appendices. Appendix \ref{appendix1} describes the construction and essential properties of $L$-completion and the associated category of $L$-complete modules, summarizing and extending \cite[Appendix A]{moravak}. Finally, we return to the height $1$ case in Appendix \ref{appendix2}, where we give a representation theoretic interpretation of and argument for the proof of the main theorem.

\subsection*{Notation and conventions}
For a fixed prime $p$ and an integer $h > 0$, height $h$ Morava $E$-theory $E_h$ is a $K(h)$-local Landweber exact ring spectrum with coefficients $E_*=W\F_{p^h}\llbracket u_1,\ldots,u_{h-1}\rrbracket[u, u^{-1}]$, where $W\F_{p^h}$ is the ring of ($p$-typical) Witt vectors on $\F_{p^h}$, $u_i$ is in degree 0 for all $i$, and $u$ has degree $2$. Note that $E_0$ is a complete Noetherian regular local ring with maximal ideal $\fm = (p, u_1, \ldots, u_{h-1})$, where we write $u_0=p$ for notational convenience. By a theorem of Goerss, Hopkins, and Miller, Morava $E$-theory has the structure of an $\mathbb{E}_{\infty}$-ring spectrum. For simplicity, the letter $E$ will always denote this theory. Furthermore, for an $E$-module $M$ (or in fact any spectrum), denote the $K(h)$-localization map by $j \colon M \to L_K M$, writing $L_K$ for Bousfield localization with respect to Morava $K$-theory $K(h)$. Note that, for any $E$-module $M$, $L_KM$ is equivalent to the Bousfield localization of $M$ with respect to $K(h)$ internal to $\Mod_E$, see~\cite[2.2]{filteredcolim}.

Let $\Si \colon \Mod_E \to \Mod_E$ be the usual suspension functor $\Si X = S^1 \sm X$. Let $\Si \colon \Mod_{E_*} \to \Mod_{E_*}$ denote the corresponding algebraic suspension functor $\Si M := E_*(S^1) \ot_{E_*} M$, which satisfies $\pi_* (\Si X) \cong \Si (\pi_* X)$. Although $E$ is a $2$-periodic cohomology theory, periodicity will not play an essential role in this paper. Therefore, we will not use the $\Z/2$-graded formalism proposed in \cite[\S 1.5, \S 2.6]{congruence}, but rather view $E_*$ as a $\Z$-graded ring and all $E_*$-modules as $\Z$-graded as well. As in \cite[\S 3]{koszul}, we will keep track of suspensions whenever necessary.

Throughout this paper, we will be mostly working on the level of homotopy categories, and all our homotopical constructions are understood to be derived. In case we have to carry out actual point-set level constructions, we can either work with $S$-modules as given \cite{ekmm} or within the $\infty$-category of spectra as developed in \cite{ha}. 

Moreover, we often abbreviate notation by omitting the subscripts of smash products, tensor products, and so on when the base ring (spectrum) is understood.

\subsection*{Acknowledgments} 

We are very grateful to Charles Rezk for introducing us to the problem that motivated this paper and thank him for many insightful conversations. Furthermore, we would like to thank Omar Antol\'in Camarena, Andy Baker, Dan Christensen, Paul Goerss, Frank Gounelas, Mike Hopkins, Mark Hovey, Jim McClure, Nat Stapleton, and Donald Yau for helpful discussions on the subject matter of this paper. We also thank the referee for several useful comments.

The first author was partially supported by an ERP scholarship and Harvard University.  The second author was partially supported by an FQRNT Postdoctoral Research Scholarship, held at the University of Illinois at Urbana-Champaign under the supervision of Charles Rezk.

\section{Preliminaries and Lazard's theorem}\label{prelim}

We will need the analogues in topology of familiar notions from algebra. For $R$ an associative ring spectrum, we introduce various categories of $R$-modules and, if $R$ is also commutative, commutative $R$-algebras, recall the construction of free commutative $R$-algebras, and then specialize to Morava $E$-theory. In Section \ref{lazard}, we prove Lazard's characterization of flat modules in the setting of module spectra over an arbitrary $\mathbb{E}_1$-ring spectrum.

\subsection{Free commutative algebras}\label{prelim1}

Consider an $\mathbb{E}_{\infty}$-ring spectrum $R$.

\begin{note}
Let $\Mod_R$ denote the category of $R$-modules, which is symmetric monoidal with respect to the smash product $\sm_R$; if the base ring is clear from context, we simply write $\wedge$. Let $\Alg_R$ denote the category of commutative $R$-algebra, i.e., commutative monoid objects in $(\Mod_R,\sm_R)$.
\end{note}

The forgetful functor $\Alg_R \to \Mod_R$ has a left adjoint $\D \colon \Mod_R \to \Alg_R$, the free commutative $R$-algebra functor, given by $\D(M) = \bigvee_{n \geq 0} (M^{\sm_R n})/{\Si_n}$. We also denote the associated forget-of-free monad by $\D \colon \Mod_R \to \Mod_R$. We write $\D = \D^R$ to specify the base ring spectrum $R$ if needed. The adjunction is monadic; in fact, the category $\Alg_R$ is isomorphic to the category of $\D$-algebras in $\Mod_R$. 

To describe the homotopical behavior of $\D$, let us recall some facts about the EKMM model of $S$-modules. First, $\E$-ring spectra and their modules can be functorially replaced by weakly equivalent commutative $S$-algebras and their mo\-du\-les \cite[II.3.6]{ekmm}. A commutative $S$-algebra $R$ can be replaced by a weakly equivalent q-cofibrant commutative $S$-algebra, with an equivalent derived category $h\Mod_{R}$ \cite[III.4.2]{ekmm}. If $R$ is a q-cofibrant commutative $S$-algebra, and $M$ is a cell $R$-module, then extended powers agree with symmetric powers: the map $(M^{\wedge n})_{h \Si_n} \to (M^{\wedge n})/\Si_n$ is a homotopy equivalence \cite[III.5.1]{ekmm}, motivating the following definition.

\begin{defn}
Let $n \geq 0$ be a non-negative integer and $M$ an $R$-module. The \Def{$n^{\text{th}}$ extended power} of $M$ is the $R$-module
\[
\D_n M = ( \overbrace{M \sm_R \ldots \sm_R M}^{n \text{ times}}_{} )_{h \Si_n}
\]
where the symmetric group $\Si_n$ acts on the $n$-fold smash product by permuting the factors, and the subscript denotes the homotopy orbit. This homotopy colimit can be computed as $\D_n M = (E \Si_n)_+ \sm_{\Si_n} M^{\sm_R n}$, where $E \Si_n$ denotes the standard contractible space with a free right $\Si_n$-action \cite[I \S 2]{hinfty}.

Extended powers form a functor $\D_n \colon \Mod_R \to \Mod_R$, which moreover passes to the homotopy category. The symmetric algebra functor $\D \colon \Mod_R \to \Mod_R$ induces a total left derived functor, also denoted $\D \colon h\Mod_R \to h\Mod_R$, which is given by\[
\D M = \bigvee_{n \geq 0} \D_n M = \bigvee_{n \geq 0} (M^{\sm_R n})_{h \Si_n}.
\]
\end{defn}

This functor defines a monad on the homotopy category of $R$-modules. An algebra structure on a spectrum $Y$ for the monad $\D \colon h\Mod_R \to h\Mod_R$ is an $\HH$ $R$-algebra structure; the corresponding structure maps are denoted by $\xi_n\colon \D_nY \to Y$. Any commutative $R$-algebra naturally becomes (upon passing to the homotopy category $h\Mod_R$) an algebra for the monad $\D$. 

\begin{defn}
An $R$-module $M$ is called \Def{finitely generated} if $\pi_* M$ is a finitely generated $R_*$-module. The module $M$ is called \Def{finitely generated and free}, or \Def{finite free} for short, if $\pi_* M$ is a finitely generated free $R_*$-module. This holds if and only if $M$ is equivalent to a finite wedge $\bigvee_{i=1}^k \Si^{d_i} R$. Let $\Modff_R$ denote the full subcategory of $\Mod_R$ consisting of finite free $R$-modules, and $h\Modff_R$ its homotopy category, viewed as a full subcategory of $h\Mod_R$.
\end{defn}

In general, $\D_n$ does not preserve finite free modules, but we will see shortly a version of it which does, for the ring spectrum of interest here, Morava $E$-theory. Recall from the introduction that we are going to fix a height $h$ for the remainder of the paper and agree to write $E=E_h$. 

\begin{prop}
\cite[3.8]{congruence}\mylabel{dmonad} The functor $\D_n\colon h\Mod_E \to h\Mod_E$ preserves $K(h)$-homology isomorphisms. In particular, the natural transformation
\[
L_K \D_n(j) \colon L_K \D_n \to L_K \D_n L_K
\]
is an isomorphism. The functor $\PP = L_K \D \colon h\Mod_E \to h\Mod_E$ admits a unique monad structure with the property that $j \colon \D \to L_K \D$ is a map of monads.
\end{prop}

In the context of $E$-modules, this \Def{completed free algebra functor} $\PP$ is better behaved than its uncompleted analogue, as the next proposition demonstrates. 

\begin{prop}\mylabel{powerff}
\cite[3.9]{congruence} If $M$ is a finite free $E$-module, then $L_K \D_n(M)$ is also finite free, for any $n \ge 0$.
\end{prop}

\begin{rmk}
\myref{powerff} would fail without the completion. That is, if $M$ is a finite free $E$-module, then $\D_n(M)$ need \textit{not} be finite free.
\end{rmk}

\subsection{Lazard's theorem for flat module spectra}\label{lazard}

The goal of this section is to generalize Lazard's characterization of flat modules to module spectra over an arbitrary $\mathbb{E}_1$-ring spectrum, inspired by conversations with Sam Raskin and Charles Rezk. All the modules and module spectra will be left modules, unless otherwise stated.

First, recall Lazard's theorem for ordinary rings. 

\begin{thm}[Lazard's theorem for flat modules]\mylabel{lazardalg}
Let $R$ be a unital and associative ring, then the following conditions are equivalent for an $R$-module $M$:
\begin{enumerate}
 \item $M$ is flat, i.e., the functor $- \otimes_R M$ from right $R$-modules to abelian groups is exact. 
 \item Every map $C \to M$ with $C$ finitely presented factors through some finitely generated free left $R$-module $F$.
 \item $M$ can be written as a filtered colimit of finitely generated free $R$-modules.  
\end{enumerate}
\end{thm}

For the rest of this section, let $R$ be an $\mathbb{E}_1$-ring spectrum, and denote by $\Mod_R$ the category of left $R$-module spectra. For technical convenience, in this subsection only, we will freely use concepts from the theory of $\infty$-categories, modeled as quasi-categories as developed in \cite{joyal} and \cite{htt}. In particular, all constructions are considered to be taking place within the setting of $\infty$-categories, so that limits and colimits are automatically derived and diagrams commute only up to homotopies witnessed by higher simplices. 

Recall that an $\infty$-category $\cC$ is filtered if and only if every map $f\colon  K \to \cC$ from a finite simplicial set $K$ admits a cocone extension $\bar{f} \colon K^{\triangleright} \to \cC$. 

\begin{lemma}\mylabel{filteredcrit}
Suppose $p\colon  \cD \hookrightarrow \cC$ is a full subcategory of a filtered $\infty$-category $\cC$. For any $c \in \cC$, let $\cD_{c/}$ be the comma $\infty$-category constructed as the pullback
\[\xymatrix{\cD_{c/} \ar[r] \ar[d] & \cC_{c/} \ar[d]\\
\cD \ar[r]_p & \cC.} \]
If for every $c \in \cC$ there exists a map $f \colon c \to d$ with $d \in \cD$, then $\cD_{c/}$ is filtered. 
\end{lemma}
\begin{proof}
Let $g \colon K \to \cD_{c/}$ be a finite diagram, which is the same as a diagram $g \colon K \to \cC_{c/}$ with all vertices in $\cD$. By adjunction, this is equivalent to a diagram $h \colon {}^{\triangleleft}K \to \cC$ with initial vertex $c$ and all but the initial vertices in $\cD$. Since $\cC$ is filtered, $h$ admits an extension $\bar{h}' \colon {}^{\triangleleft}K^{\triangleright} \to \cC$ with final vertex $c' \in \cC$. The assumption implies the existence of $c' \to d$ with $d \in \cD$, giving another extension $\bar{h} \colon {}^{\triangleleft}K^{\triangleright} \to \cC$ with final vertex $d$. Using adjunction again, we get the desired extension $\bar{g}\colon  K^{\triangleright} \to \cD_{c/}$ of the given diagram $g$. 
\end{proof}

The category $\Comp_R$ of \Def{perfect} modules is by definition the smallest full stable subcategory of $\Mod_R$ which contains $R$ and is closed under retracts. By \cite[8.2.5.2]{ha}, it can be identified with the full subcategory of compact objects, and 
\[\mathrm{Ind}\Comp_R = \Mod_R\] 
i.e., every $M \in Mod_R$ is equivalent to the colimit of the canonical filtered diagram $(\Comp_R)/M \to \Mod_R$. Moreover, \cite[8.2.5.7]{ha} exhibits a perfect pairing between perfect left $R$-modules and perfect right $R$-modules, 
\[ \LComp_R \times \RComp_R \xrightarrow{-\wedge_R-} \Sp \xrightarrow{\Omega^{\infty}} \Top;\]
in particular, if $C \in \Comp_R$, we denote its \Def{dual} by $\Hom(C,R) \in \Comp_R$, viewed as a right $R$-module.

\begin{thm}[Lazard's theorem for flat module spectra] \mylabel{lazardtop}
For $M \in \Mod_R$, the following conditions are equivalent:
\begin{enumerate}
 \item $\pi_*M$ is flat as a graded $\pi_*R$-module.
 \item Every map $C \to M$ with $C \in \Comp_R$ factors through some $F \in \Modff_R$.
 \item $M$ can be written as a filtered colimit of finite free $R$-modules.
\end{enumerate}
Note that conditions (2) and (3) are interpreted within the $\infty$-category $\Mod_R$.
\end{thm}

The following argument provides, with the obvious modifications, also a proof of \myref{lazardalg}.

\begin{proof}
(1) $\Longrightarrow$ (2). Since every perfect module $C$ is a retract of a finite cell $R$-module, it follows by induction on the number of cells that the natural map of spectra
\[\Hom(C,R) \wedge M \to \Hom(C,M)\]
is an equivalence for any $M \in \Mod_R$. Assuming $M$ satisfies (1), we claim that, for every $C \in \Comp_R$, the natural map 
\[\pi_*\Hom(C,R) \otimes_{\pi_*R}\pi_*M \to \pi_*\Hom(C,M)\] 
is an isomorphism; this immediately implies (2). Because perfect modules are dualizable, taking $P = \Hom(C,R)$, this is equivalent to the claim that $f_*\colon \pi_*P \otimes_{\pi_*R} \pi_*M \to \pi_*(P \wedge M)$ is an isomorphism for any $P \in \Comp_R$. 

Denote by $\cP$ the full subcategory of $\Mod_R$ on those perfect right modules $P$ for which $f_*$ is an isomorphism. It is clear that $\cP$ contains $R$ and is closed under finite coproducts, shifts, and retracts, so it suffices to show that it is closed under extensions. To this end, let $P' \to P \to P''$ be a fiber sequence of $R$-modules with $P', P'' \in \cP$. Since $\pi_*M$ is flat over $\pi_*R$, we get a map between 5-term exact sequences of $\pi_*R$-modules
\[\resizebox{\textwidth}{!}{\xymatrix{\pi_{*+1}P'' \otimes_{\pi_*R} \pi_*M \ar[r] \ar[d]^\sim & \pi_*P' \otimes_{\pi_*R} \pi_*M \ar[r] \ar[d]^{\sim} & \pi_*P \otimes_{\pi_*R} \pi_*M \ar[r] \ar[d]^{\sim} \ar[d]_{\therefore} & \pi_*P'' \otimes_{\pi_*R} \pi_*M \ar[r] \ar[d]^{\sim} & \pi_{*-1}P' \otimes_{\pi_*R} \pi_*M \ar[d]^{\sim} \\
\pi_{*+1}(P'' \wedge M) \ar[r] & \pi_{*}(P' \wedge M) \ar[r] & \pi_{*}(P \wedge M) \ar[r] & \pi_{*}(P'' \wedge M) \ar[r]  & \pi_{*-1}(P' \wedge M)}}\]
and the five lemma gives the claim. 

(2) $\Longrightarrow$ (3). We will show that the canonical diagram $(\Modff_R)/M \to \Mod_R$ is filtered and has colimit $M$. Every finitely generated free $R$-module is compact, so there is an inclusion as a full subcategory 
\[p\colon  (\Modff_R){/M} \hookrightarrow (\Comp_R)/M.\]
Assumption (2) corresponds precisely to the condition in \myref{filteredcrit}, so the comma $\infty$-category $((\Modff_R){/M})_{N/}$ is filtered for every $N \in (\Comp_R)/M$, hence weakly contractible by \cite[5.3.1.18]{htt}. Therefore, \cite[4.1.3.1]{htt} shows that $p$ is final, thus the colimit of 
\[(\Modff_R)/M \xrightarrow{p} (\Comp_R)/M \to \Mod_R\]
is $M$. Moreover, applying \myref{filteredcrit} to the special case $N=0$ gives that this diagram is filtered. 

(3) $\Longrightarrow$ (1). It is obvious that (1) is true for $M=R$, so the claim follows since homotopy groups commute with filtered colimits. 
\end{proof}

\begin{defn}
An $E$-module $M$ is called \Def{flat} if it satisfies the equivalent conditions of \myref{lazardtop}. The full subcategory on the flat $E$-modules is denoted by $\Modflat_E$.
\end{defn}

\begin{rmk}
\myref{lazardtop} is a generalization of both Lurie's version of Lazard's theorem for connective modules over a \emph{connective} associative ring spectrum $R$ \cite[8.2.2.15]{ha}, and Rezk's result \cite[3.7]{congruence}, which is the special case of the above for Morava $E$-theory. Note, however, that Lurie's definition of flatness differs slightly from ours, in that he defines, for $R$ connective, a flat module to be an $R$-module $M$ with $\pi_0M$ flat over $\pi_0R$ and such that the extra condition
\[ \xymatrix{\pi_*R \otimes_{\pi_0R}\pi_0M \ar[r]^-{\simeq} & \pi_*M}\]
holds. The proof given above can be adapted to this case as well.
\end{rmk}

\section{Algebraic approximation functors}\label{aaf}

After a brief review of Rezk's algebraic approximation functors, we study their completed analogues and state our main theorem. At the end of this section, we discuss some future applications.

\subsection{Rezk's algebraic approximation functors}

We recall Rezk's construction of algebraic approximation functors in \cite[\S 4]{congruence}, along with their main properties. For this, we will make frequent use of:

\begin{prop} \mylabel{equivcatff}
\cite[3.6]{congruence} The functor $\pi_* \colon h\Modff_E \to \Modff_{E_*}$ is an equivalence of categories. 
\end{prop}

The algebraic approximation functors $\T_n \colon \Mod_{E_*} \to \Mod_{E_*}$ are constructed so as to capture the algebraic structure present on $\pi_* L_K \D_n$. In particular, they will satisfy $\T_n(\pi_*E) \cong \pi_* L_K \D_n (E)$ and a similar formula for finitely generated free $E_*$-modules.

Consider the diagram of categories
\begin{equation} \label{diagtn}
\xymatrix{
h\Modff_E \ar[d]_{\pi_*}^{\sim} \ar[r]^i & h\Mod_E \ar[d]_{\pi_*} \ar[r]^-{L_K \D_n} & h\Mod_{E} \ar[d]^{\pi_*} \\
\Modff_{E_*} \ar[r]_j & \Mod_{E_*} \ar@{-->}[r]_{\T_n} & \Mod_{E_*} \\}
\end{equation}
where the left-hand square commutes (strictly), the functors $i$ and $j$ are inclusions of full subcategories, and the downward arrow $\pi_*$ on the left is an equivalence of categories, by \myref{equivcatff}. Note that the functor $\T_n$ constructed below does \textit{not} make the right-hand square commute.

In order to define the algebraic approximations functor $\T$ we have to briefly recall the concept of (left) Kan extension, see \cite[$\S$ X.3]{working} or \cite[1.1]{riehl}. Let $F\colon  \cC \to \cE$ and $K\colon  \cC \to \cD$ be functors. The left Kan extension of $K$ by $F$ consists, if it exists, of a functor $\Lan_KF\colon  \cD \to \cE$ together with a natural transformation $\eta\colon  F \to \Lan_KF \circ K$, which is the initial such pair: 
\[\xymatrix{\cC \ar[rr]^{F} \ar[rd]_{K} & \ar@{}[d]_{\Downarrow}^{\eta} & \cE \\
& \cD. \ar@{-->}[ur]_{\Lan_KF}  
}\]
Universality is in the following sense: If $(G\colon  \cD \to \cE, \theta\colon  F \to G \circ K)$ is another pair as above, then there exists a unique natural transformation $\omega\colon  \Lan_K F \to G$ such that $\theta = \omega K \circ \eta$. In other words, if ${\cE}^{\cD}=\mathrm{Fun}(\cD,\cE)$ and similarly for ${\cE}^{\cC}$, then universality gives a bijection
\[{\cE}^{\cD}( \Lan_K F,G) = {\cE}^{\cC}(F,GK)\]
for any $G \in {\cE}^{\cD}$.

Assuming that $\cC$ is essentially small and $\cE$ is cocomplete, the left Kan extension of $F$ along $K$ always exists and can be computed pointwise using the following formula:
\[ \Lan_KF(d) = \int^{c \in \cC} \Hom_{\cD}(Kc,d) \cdot Fc.\]
Here, the integral sign denotes the coend and \[( - \cdot - )\colon  \mathrm{Set} \times \cE \to \cE, \ (S,e) \mapsto \coprod_Se\] 
is the copower. Note that, if $\cE$ is cocomplete and $K$ is fully faithful, then $\eta$ is a natural isomorphism, so $\Lan_KF \circ K = F$.

\begin{defn} \mylabel{deftn}
\cite[\S 4.2]{congruence} For every $n \geq 0$, the \Def{algebraic approximation functor} $\T_n \colon \Mod_{E_*} \to \Mod_{E_*}$ is defined as
\[
\T_n = \Lan_{\pi_* i} \left( \pi_* L_K \D_n i \right)
\]
i.e., the left Kan extension of the functor $\pi_* L_K \D_n i \colon h \Modff_E \to \Mod_{E_*}$ along the functor $\pi_* i = j \pi_* \colon h\Modff_E \to \Mod_{E_*}$.

We define the functor $\T \colon \Mod_{E_*} \to \Mod_{E_*}$ as the direct sum $\T = \oplus_{n \geq 0} \T_n$.
\end{defn}

Note that this left Kan extension exists, so that $\T_n$ is well defined. Indeed, the category $h \Modff_E$ is essentially small, being equivalent to the essentially small category $\Modff_{E_*}$, and the target category $\Mod_{E_*}$ is cocomplete.

\begin{note}
Using the universal property of left Kan extensions and $\Hom(iM,N)$ $= \Hom(\pi_*iM,\pi_*N)$ for $M \in h\Modff_{E}$ and $N \in h\Mod_E$, one obtains a natural transformation 
\[ \al_n \colon \T_n (\pi_* M) \to \pi_* (L_K \D_n M),\]
called the \Def{approximation map} in \cite[\S 4.3]{congruence}, and likewise
\[ \al = \oplus_{n \geq 0}\al_n \colon \T (\pi_* M) \to \pi_* (L_K \D M).\]
\end{note}

\begin{warning}
In \cite[3.2]{koszul}, the approximation map is defined as a natural transformation $\T (\pi_* L_K M) \to \pi_* (L_K \D M)$ for all $E$-module $M$. We denote this version of the approximation map by $\tilde{\al}$ instead, to avoid ambiguity. However, $\tilde{\al}$ can be thought of as a special case of $\al$, namely for the $E$-module $L_K M$, using the natural equivalence $L_K \D \ral{\sim} L_K \D L_K$, as illustrated in the commutative diagram:
\[
\xymatrix{
\T \left( \pi_* L_K M \right) \ar[dr]_{\al_{L_K M}} \ar[r]^{\tilde{\al}_M} & \pi_* L_K \D M \ar[d]^{\cong} \\
& \pi_* L_K \D L_K M. \\
} 
\]
\end{warning}

\begin{rmk}
The comparison map $\tilde{\alpha}\colon  \T \pi_*L_K \to \pi_* L_K \D$ induces a lifting of $\pi_*L_K\colon $ $ h\Mod_E \to \Mod_{E_*}$ to a functor
\[\xymatrix{h\Alg_{\D} \ar@{-->}[r]^-{\pi_*L_K} \ar[d]_U &  \Alg_{\T} \ar[d]^U \\
h\Mod_{E} \ar[r]_-{\pi_*L_K} & \Mod_{E_*},}\]
where the arrows labelled $U$ are the natural forgetful functors. This allows the study of power operations through the category of $\T$-algebras.
\end{rmk}

\begin{prop}\mylabel{comparison}
\cite[4.4]{congruence} If $M$ is a finite free $E$-module, then the map $\al_n \colon \T_n (\pi_* M) \to \pi_* (L_K \D_n M)$ is an isomorphism.
\end{prop}

Recall that a small category $I$ is called \Def{sifted} if colimits over $I$ commute in $\Set$ with finite products. A \Def{sifted colimit} is a colimit of a diagram over a sifted category. For example, filtered colimits and reflexive coequalizers are sifted colimits. 

\begin{thm}\mylabel{tmonad}
\cite[4.5]{congruence} The functor $\T \colon \Mod_{E_*} \to \Mod_{E_*}$ defined in \myref{deftn} admits the structure of a monad, compatible with that of $L_K \D$.
\end{thm}

\begin{proof}
By definition, for any $M \in \Mod_{E_*}$,
\[ \T_n(M) = \Lan_{\pi_*i} \pi_*L_K\D_n i (M)=  \int^{F \in h\Modff_E} \mathrm{Hom}(\pi_*i(F),M) \cdot \pi_*L_K\D_n i(F)\]
because the left Kan extension can be constructed pointwise. Since any $F \in h\Modff_E$ is small and projective in $\Mod_{E_*}$, $\mathrm{Hom}(\pi_*iF,-)$ commutes with sifted colimits in $\Mod_{E_*}$, or equivalently, it commutes with filtered colimits and reflexive coequalizers; see \cite[2.1]{siftedcolimits}. This implies that the left Kan extension $\T_n$ also commutes with sifted colimits. Therefore, the monad structure on $\T = \oplus_{n \geq 0} \T_n$ is determined by its restriction to $h\Modff_E$, where $\T_n$ coincides with $\pi_*L_K\D_n$ by \myref{comparison}. By virtue of \myref{dmonad} and \myref{powerff}, the claim follows.
\end{proof}

Moreover, $\T$ has the structure of a graded exponential monad. Roughly speaking, a monad $\mathbb{M}$ on a category $\cC$ is \Def{exponential}, if it is symmetric monoidal with respect to two symmetric monoidal structures on $\cC$. If, additionally, the monad admits a decomposition $\mathbb{M} = \bigoplus_{n \ge 0} \mathbb{M}_n$ into endofunctors $\mathbb{M}_n$ that is compatible with the symmetric monoidal structure, then $\mathbb{M}$ is called \Def{graded exponential}. The reader interested in the precise definition is referred to \cite[2.2]{koszul}; for our purposes, the following result will be sufficient.

\begin{thm}\mylabel{exp}
\cite[4.8]{congruence} The monad $\T \colon (\Mod_{E_*},\op,0) \to (\Mod_{E_*},\ot,E_*)$ is graded exponential. In particular, if $M, N \in \Mod_{E_*}$ and $n \ge 0$, then there is a natural isomorphism
\[
 \bigoplus_{i+j =n} \T_i(M) \otimes \T_j(N)  \xrightarrow{\cong} \T_n(M \oplus N)
\]
where the coproduct is taken over all pairs $(i,j)$ of non-negative integers which sum to $n$. 
\end{thm}

The properties of the monad $\T$ imply that it preserves various subcategories of $\Mod_{E_*}$, which will become important in Section \ref{proofofthm}. 

\begin{prop}\mylabel{tpres}
\cite[3.9, 4.4, 4.6]{congruence} The algebraic approximation functor $\T_n$ preserves the categories $\Modff_{E_*}$ and $\Modflat_{E_*}$ of finite free modules and flat modules, respectively. 
\end{prop}

\begin{cor}
If $F \in \Mod_{E_*}$ is free, then so is $\T(F)$. In particular, $\T_n(F)$ is projective for all $n \ge 0$. 
\end{cor}
\begin{proof} A pointed module is a pair $(M,f)$ consisting of a module $M$ together with a morphism $f\colon E_* \to M$. If $M$ is free and $f$ is a split monomorphism (which implies that $\coker f$ is free, since $E_*$ is local), then $(M,f)$ is called pointed free. Let $M_i = (M_i,f_i)$ be a set of pointed free modules indexed by $I$, and consider
\[ \bigotimes_{i \in I}M_i := \colim \bigotimes_{i \in I^{\mathrm{fin}}} M_i \]
where the colimit diagram is indexed by the finite subsets $I^{\mathrm{fin}} \subset I$, and the maps in the diagram are defined by inserting basepoints $f_i \colon E_* \to M_i$. If $\{b_{i,j} \mid j \in S_i\}$ is a basis of $\coker f_i$, then $\{ b_{i_1,j_1} \otimes \cdots \otimes b_{i_k,j_k} \mid k \geq 0, i_t \in I \text{ are distinct}, j_t \in S_{i_t} \}$ is a basis of $\bigotimes_{i \in I}M_i$, which is therefore also free. In other words, $\bigotimes_{i \in I}M_i$ is pointed free on the weak product of the pointed sets $S_i \cup \{\ast\}$.  Taking $F = \colim_I F_i$ a free $E_*$-module written as a filtered colimit of finite free modules, then we get 
\[\begin{array} {lcl} 
\T(F) & = & \T(\colim \bigoplus_{i \in I^{\mathrm{fin}}} E_*) \\ 
& = & \colim\ \T(\bigoplus_{i \in I^{\mathrm{fin}}} E_*) \\
& = & \colim \bigotimes_{i \in I^{\mathrm{fin}}} \T(E_*) \\
& = & \bigotimes_I \T(E_*),
\end{array}\]
using that $\T(M)$ is naturally pointed by $E_*= \T(0) \to \T(M)$. Since $\bigotimes_I \T(E_*)$ is free by the above argument, the claim follows. 
\end{proof}

\begin{rmk}
The construction and basic properties of $\T$ as well as the existence of the approximation map $\al$ are entirely formal. In fact, one can consider a functor on a category of monads with certain finiteness conditions
\[\Lan_{\pi_*i}\pi_*(-)i\colon \Monad(h\Mod_R) \to \Monad(\Mod_{R_*}), \]
which exists for any appropriate ring spectrum $R$. Assuming $R$ is $\mathbb{E}_{\infty}$ and satisfies an analogue of \myref{powerff}, applying this functor to the free commutative $R$-algebra monad $\D^R$ we obtain an algebraic approximation $\T^R$ for the theory of power operations on $R$, together with an approximation map
\[ \al^R \colon \T^R (\pi_* M) \to \pi_* (\D^R M) \]
satisfying the analogue of \myref{comparison}. In particular, this formalism can be applied to the (Koszul dual) free Lie-algebra monad $\Lie\colon \Mod_E \to \Mod_E$ to obtain algebraic approximation functors $\T_{\Lie}$ and a ring of additive Lie-type power operations $\Gamma_{\Lie}$.
\end{rmk}

\subsection{Completed algebraic approximation functors}

In this subsection, we construct an endofunctor $\That \colon \Modhat_{E_*} \to \Modhat_{E_*}$ which better approximates the algebraic structure found on $\pi_* L_K \D$. More specifically, $\That$ takes into account the fact that $\pi_* L_K \D M$ is $L$-complete.

To begin, note that $E_*$ is $L$-complete, since $E$ is $K(h)$-local, and thus all finite free $E_*$-modules are $L$-complete. In other words, the functor $\pi_* i = j \pi_* \colon h\Modff_E \to \Mod_{E_*}$ lands in the full subcategory $\Modhat_{E_*}$. Likewise, we will see in \myref{homotknlocal} that $\pi_* L_K \D_n M$ is $L$-complete for any $E$-module $M$, and so the functor $\pi_* L_K \D_n \colon h\Mod_E \to \Mod_{E_*}$ takes values in $\Modhat_{E_*}$ as well. 

\begin{prop}\mylabel{algtotopss}
\cite[2.3]{filteredcolim} For any $E$-module $M \in \Mod_E$, there exists a natural, conditional and strongly convergent spectral sequence of $E_*$-modules 
\[ E^2_{s,t}=(L_s\pi_*M)_t \Rightarrow \pi_{s+t}L_KM \]
with $E_2^{s,t}=0$ if $s>h$. 
\end{prop}

\begin{rmk}
For any spectrum $X$, the \Def{completed $E$-homology} of $X$ is defined as
\[E^{\wedge}_* X \colon = \pi_* L_K (E \sm X).\]
The spectral sequence \myref{algtotopss} for the $E$-module $M = E \sm X$ can then be written as
\[E^2_{s,t}=(L_s E_*X)_t \Rightarrow E^{\wedge}_{s+t}X.\]
Note that $E^{\wedge}_*$ is not technically a homology theory, as it fails to preserve coproducts and filtered homotopy colimits. Hovey shows in \cite[1.1, 2.4]{filteredcolim} how this failure can be measured by the higher derived functors $L_s$ for $1 \leq s \leq h$. Informally speaking, \emph{completed Morava $E$-theory is $h$ derived functors away from being a homology theory}.
\end{rmk}

\begin{cor}\mylabel{homotknlocal}
An $E$-module $M$ is $K(h)$-local if and only if $\pi_* M$ is an $L$-complete $E_*$-module. In particular, the functor $\pi_*L_K\D_n$ factors through $\Modhat_{E_*}$,
\[\xymatrix{h\Mod_E \ar[r]^{L_K\D_n} \ar@{-->}[d] & h\Mod_E \ar[d]^{\pi_*} \\
\Modhat_{E_*} \ar[r]_{\iota} & \Mod_{E_*}. }\]
\end{cor}

\begin{proof}
For the ``only if'' direction, the $E^2$-term of \myref{algtotopss} consists of $L$-complete modules, by \myref{lprop}. Therefore the abutment is also $L$-complete, since $L$-complete modules form an abelian subcategory of $\Mod_{E_*}$, by \myref{lcompletecat}.

Since $\pi_*(L_K M)$ is $L$-complete, the map $\pi_*M \to \pi_*(L_K M)$ uniquely factors as
\[
\pi_* M \ral{\eta} L_0 (\pi_* M) \to \pi_*(L_K M)
\]
where the second map is the (left) edge morphism in the spectral sequence. If $\pi_* M$ is $L$-complete, then by \myref{lprop} we have $L_s \pi_* M = 0$ for all $s > 0$ and thus the spectral sequence collapses to the isomorphism $L_0 (\pi_* M) \cong \pi_*(L_K M)$. This proves the ``if'' direction.
\end{proof}

\begin{rmk}
The special case where $M$ is of the form $E \sm X$ for some spectrum $X$ is proved in \cite[8.4(a)]{moravak} using a different argument.
\end{rmk}

Hence the construction of $\T_n$ as a left Kan extension in \myref{deftn} can be refined as follows.  Consider the following diagram of categories, in which only the diagram of solid arrows is commutative: 
\begin{equation}\label{diagtnhat} 
\xymatrix{
h\Modff_E \ar[dd]_{\pi_*}^{\sim} \ar[dr] \ar[rr]^i & & h\Mod_E \ar[dd]_{\pi_*} \ar[dr]^{\pi_* L_K \D_n} \ar[rr]^{L_K \D_n} & & h\Mod_{E} \ar[dd]^{\pi_*} \\
& \Modhat_{E_*} \ar@{-->}@/_1pc/[rr]_(0.7){\That_n} \ar[dr]_{\iota} & & \Modhat_{E_*} \ar[dr]_{\iota} & \\
\Modff_{E_*} \ar[ur]_{j'} \ar[rr]_j & & \Mod_{E_*} \ar@{-->}[rr]_{\T_n} & & \Mod_{E_*}. \\
}
\end{equation}

\begin{defn} For every $n \geq 0$, the \Def{completed algebraic approximation functor} $\That_n \colon \Modhat_{E_*} \to \Modhat_{E_*}$ is defined as
\[
\That_n = \Lan_{j' \pi_*} (\pi_* L_K \D_n i)
\]
i.e., the left Kan extension of the functor $\pi_* L_K \D_n i \colon h \Modff_E \to \Modhat_{E_*}$ along the functor $j' \pi_* \colon h\Modff_E \to \Modhat_{E_*}$, as illustrated in the diagram \eqref{diagtnhat}.

As in \myref{deftn}, this Kan extension exists since $h\Modff_E$ is essentially small and $\Modhat_{E_*}$ is cocomplete.
\end{defn}

Using the natural isomorphism $L_0 \iota = \id \colon \Modhat_{E_*} \to \Modhat_{E_*}$, the functor $(\pi_* L_K \D_n) \colon$  $h\Mod_E \to \Modhat_{E_*}$ can be written as the composite
\[
h\Mod_E \ral{L_K \D_n} h\Mod_E \ral{\pi_*} \Mod_{E_*} \ral{L_0} \Modhat_{E_*}
\]
and likewise for the functor $j' \pi_* \colon h\Modff_E \to \Modhat_{E_*}$
\begin{align*}
j' \pi_* &= L_0 \iota j' \pi_* \\
&= L_0 j \pi_*.
\end{align*}
Hence, the diagram defining $\That_n$ can be rewritten as 
\[
\xymatrix{
h\Modff_E \ar[d]^{\pi_*}_{\sim} \ar[r]^i & h\Mod_E \ar[d]^{\pi_*} \ar[r]^{L_K \D_n} & h\Mod_{E} \ar[d]^{\pi_*} \\
\Modff_{E_*} \ar[r]_j & \Mod_{E_*} \ar@{-->}[r]_{\T_n} \ar[d]_{L_0} & \Mod_{E_*} \ar[d]^{L_0} \\
& \Modhat_{E_*} \ar@{-->}[r]_{\That_n} & \Modhat_{E_*}}
\]
which yields the following alternate description of $\That_n$.

\begin{prop}
The completed algebraic approximation functor satisfies $\That_n = L_0 \T_n \iota \colon \Modhat_{E_*} \to \Modhat_{E_*}$.
\end{prop}

\begin{proof}
Recalling that $L_0$ is left adjoint to $\iota$, we obtain natural isomorphisms:
\begin{align*}
\That_n &= \Lan_{j' \pi_*} \left(\pi_* L_K \D_n i \right) \\
&= \Lan_{L_0 j \pi_*} \left( L_0 \pi_* L_K \D_n i \right) \\
&= \Lan_{L_0} \Lan_{j \pi_*} \left( L_0 \pi_* L_K \D_n i \right) \\
&= \Lan_{j \pi_*} \left( L_0 \pi_* L_K \D_n i \right) \iota \\
&= L_0 \Lan_{j \pi_*} \left( \pi_* L_K \D_n i \right) \iota \\
&= L_0 \T_n \iota. \qedhere
\end{align*}
\end{proof}

The completed algebraic approximation functor $\That_n$ resembles the structure present on the homotopy groups of $K(h)$-local commutative $E$-algebras more closely than $\T_n$, as can be seen by comparing the next lemma with \myref{comparison}; see also \myref{flatchar}.

\begin{prop}\mylabel{flatcomparison} \cite[4.9]{congruence}
If $M$ is a flat $E$-module, then the natural map $\That_n \pi_*M \to \pi_*\LD_nM$ is an isomorphism for any $n$. 
\end{prop}

We are now ready to state the main theorem of this paper, which will be proved in Section \ref{proofofthm}.

\begin{thm}\mylabel{mainthm}
The functor $\T \colon \Mod_{E_*} \to \Mod_{E_*}$ preserves $L_0$-equivalences, which means that the natural map
\[
L_0 \T (M) \ral{L_0 \T \eta} L_0 \T L_0 (M) 
\]
is an isomorphism for all $E_*$-module $M$.
\end{thm}

By \myref{mainthm} and \myref{CompletedMonad}\eqref{InheritedMonad}, we obtain the following.

\begin{cor}
The completed algebraic approximation functor $\That=L_0 \T \iota \colon \Modhat_{E_*} \to \Modhat_{E_*}$ admits a natural monad structure, inherited from that of $\T$.
\end{cor}

\begin{rmk}
Note that the argument given in the proof of \myref{tmonad} does \emph{not} work to show that the completed algebraic approximation functors $\That$ form a monad. Indeed, $F \in h\Modff_{E}$ is not small when viewed as an object of $\Modhat_{E_*}$ via $L_0\pi_*i$, see \myref{small}. Thus
\[ \mathrm{Hom}(\pi_*iF,-)\colon  \Modhat_{E_*} \to \Set\]
does not preserve filtered colimits, so we cannot reduce to finite free modules as before. Instead, it is a consequence of our main theorem that $\That$ commutes with sifted colimits.
\end{rmk}

\begin{cor}
The functor $\That_n$ commutes with sifted colimits. 
\end{cor}
\begin{proof}
We need to show that $\That_n$ commutes with reflexive coequalizers and filtered colimits. The first part of the statement follows from the construction, since finite free $E_*$-modules are complete and hence projective in $\Modhat_{E_*}$, by \myref{flatchar}.

For the second claim, consider a filtered diagram $D\colon  \cC \to \Modhat_{E_*}$ and denote by $\mathrm{colim} \, \iota D$ the colimit of the corresponding diagram (via $\iota$) in $\Mod_{E_*}$; then we have
\[\begin{array} {lcl} 
\That_n \lcolim \, D & = & L_0 \T_n\iota L_0 \colim \, \iota D \\ 
& = & L_0 \T_n \colim \, \iota D \\
& = & L_0 \colim \, \T_n\iota D \\
& = & L_0 \colim \, \iota L_0 \T_n \iota D \\
& = & \lcolim \, \That_n D
\end{array}\]
by \myref{mainthm}.
\end{proof}

In view of \myref{mainthm} and \myref{CompletedMonad}\eqref{CatAlgCompletedMonad}, the homotopy of $K(h)$-local commutative $E$-algebras takes values in $\That$-algebras, yielding the functor:
\[
\pi_* \colon \Alghat_E \to \Alghat_{\T} \cong \Alg_{\That}.
\]

In \cite[4.12]{congruence}, it is shown that the forgetful functor $U \colon \Alg_{\T} \to \Alg_{E_*}$ is \Def{plethyistic}, i.e., it reflects isomorphisms and has both a left adjoint and a right adjoint. The analogous statement in the completed setting also holds.

\begin{prop}
The forgetful functor $\hat{U} \colon \Alg_{\That} \to \Alghat_{E_*}$ is plethyistic.
\end{prop}
\begin{proof}
This is proven exactly as in \cite[4.12]{congruence}, using \myref{ksmall} to assure that the adjoint functor theorem is applicable. 
\end{proof}

\subsection{A word about applications}\label{appli}

The usefulness of \myref{mainthm} becomes apparent in recent work of Rezk \cite{calculations}, where calculations of power operations are carried out. Completion can now be dealt with systematically, as it is built into the structure of $\That$-algebras, unlike $\T$-algebras. Moreover, the fact that $\That$ is a monad allows the use of standard categorical and homological machinery for algebras over monads, such as simplicial resolutions of $\That$-algebras.

One application is the computation of $\pi_* \PP M$. The monad $\T \colon \Mod_{E_*} \to \Mod_{E_*}$ was constructed to encode the algebraic structure present in $\pi_* L_K \D M = \pi_* \PP M$ for any $E$-module $M$ (which is $K(h)$-local, without loss of generality). The completed monad $\That = L_0 \T \iota \colon \Modhat_{E_*} \to \Modhat_{E_*}$ provides a better approximation. As noted in \myref{flatcomparison}, the comparison map $\That(\pi_* M) \to \pi_* \PP M$ is an isomorphism whenever $M$ is a flat $E$-module. However, the problem of describing the algebraic structure of $\pi_* \PP M$ for \emph{any} $M$ remains difficult. 

There is a spectral sequence roughly of the form
\[
E^2_{s,t} = \left( (\mathbf{L}_s \That)(\pi_* M) \right)_t \Rightarrow \pi_{s+t} \PP M
\]
where $\mathbf{L}_s \That \colon \Modhat_{E_*} \to \Modhat_{E_*}$ denotes the $s^{\text{th}}$ left derived functor of the non-additive functor $\That$, and the comparison map $\That(\pi_* M) \to \pi_* \PP M$ is an edge morphism. We hope that working with the completed monad $\That$ instead of $\T$ will facilitate the construction of this spectral sequence and its calculation.

At height $h=1$, McClure provided an explicit description of $\pi_* \PP (K \sm_S X)$ for any spectrum $X$ in terms of $K$-homology operations for $X$ \cite[IX \S 3]{hinfty}. It would be interesting to relate those results to our proposed calculation.

Another application is in computing topological Andr\'e-Quillen (co)homology of (augmented) $K(h)$-local commutative $E$-algebras. This (co)homology appears notably in work of Behrens and Rezk on the Bousfield-Kuhn functors \cite{bktaq}. Topological Andr\'e-Quillen homology of a commutative $E$-algebra $A$ is the $E$-module $\TAQ(A)$ of ``indecomposables'' of $A$, in some suitable derived sense. When working $K(h)$-locally, the object of interest is $L_K\TAQ(A)$. At the algebraic level, there is a corresponding ``indecomposables'' functor $F \colon \Alg_{\T} \to \Mod_{E_*}$ and an $L$-completed version $\hat{F} = L_0 F \iota \colon \Alg_{\That} \to \Modhat_{E_*}$. There is a spectral sequence roughly of the form
\[
E^2_{s,t} = \left( (\mathbf{L}_s \hat{F})(\pi_* A) \right)_t \Rightarrow \pi_{s+t} L_K \TAQ(A)
\]
where the $E^2$ term can be thought of as Andr\'e-Quillen homology for $\That$-algebras. At height $h=1$, $\That$-algebras are sometimes referred to as $\te$-algebras, and their Andr\'e-Quillen (co)homology played an important role in work of Goerss and Hopkins \cite[\S 2]{ghproblems} \cite{ghspaces}.

\begin{rmk}
There are other approaches for computing topological Andr\'e-Quillen homology from (algebraic) Andr\'e-Quillen homology \cite{ahsstaq}.
\end{rmk}

\section{Proof of the main theorem}\label{proofofthm}

The goal of this section is to give a proof of our main \myref{mainthm}, for any prime $p$ and height $h$. An alternative, more explicit argument in the height $1$ case will be the subject of Section \ref{height1}.

\subsection{Reduction to the key property}

Let $n$ be a fixed non-negative integer. The following result, the proof of which is deferred to the next subsection, is the key technical property of the functors $\T_n$.

\begin{prop}\mylabel{key}
There exists a positive integer $k=k(n)$ such that for all $m \le n$ and all $q \in \Z$, the natural map
\[ E_*/\fm \otimes \T_m(\Si^q E_*) \xrightarrow{} E_*/\fm \otimes \T_m(\Si^q E_*/\fm^k) \]
is an isomorphism. Note that if this holds, then the same conclusion holds for any integer $k' \geq k(n)$ in place of $k$.
\end{prop}

Note that, if the functors $\T_m$ and $E_*/\fm \otimes -$ commute, the proposition would be automatic, in which case we could take any $k$. This holds if $m < p$, since then the homotopy orbit spectral sequence collapses at the $E^2$-page. However, as we will see in Section \ref{height1}, it is wrong in general, so that \myref{key} has content and we do need $k > 1$; see \myref{nonlinear}.

\begin{cor}\mylabel{keygeneral}
There exists a positive integer $k=k(n)$ such that for all $m \le n$ and  $M \in \Mod_{E_*}$ the natural map 
\[  E_*/\fm \otimes \T_m(M) \xrightarrow{} E_*/\fm \otimes \T_m(E_*/\fm^k \otimes M) \]
is an isomorphism.
\end{cor}

\begin{proof}
First, let us assume $M= \Si^{q_1} E_* \op \cdots \op \Si^{q_r} E_*$ is a finite free module of rank $r$. Then, using that $\T$ is an exponential monad (\myref{exp}) and since $E_*/\fm \otimes E_*/{\fm} \cong E_*/{\fm}$, we get
\[ E_*/\fm \otimes \T_n(M) \cong \bigoplus_{i_1+ \cdots + i_r = n} \left( E_*/\fm \otimes \T_{i_1}(\Si^{q_1} E_*) \right) \otimes \cdots \otimes \left( E_*/\fm \otimes \T_{i_r}(\Si^{q_r} E_*) \right) \]
where the direct sum is indexed by all $r$-tuples of non-negative integers $(i_1,\ldots,i_r)$ summing to $n$. By \myref{key}, this is isomorphic to
\begin{align*}
&\bigoplus_{i_1+ \cdots + i_r = n} \left( E_*/\fm \otimes \T_{i_1}(\Si^{q_1} E_*/{\fm}^k) \right) \otimes \cdots \otimes (E_*/\fm \otimes \T_{i_r} \left( \Si^{q_r} E_*/{\fm}^k) \right) \\
\cong &E_*/\fm \otimes \T_n \left( \Si^{q_1} E_*/{\fm}^k \op \cdots \op \Si^{q_r} E_*/{\fm}^k \right) \\
\cong &E_*/\fm \otimes \T_n(E_*/\fm^k \otimes M)
\end{align*}
where the first isomorphism uses the exponential monad structure again. By Lazard's theorem, we can write any $M \in \Modflat_{E_*}$ as a filtered colimit of finite free modules $F_j$, $M = \mathrm{colim}_J F_j$, thus
\begin{align*}
E_*/\fm \otimes \T_n(M) &\cong \mathrm{colim}_J (E_*/\fm \otimes \T_n(F_j)) \\
&\cong \mathrm{colim}_J (E_*/\fm \otimes \T_n(E_*/\fm^k \otimes F_j)) \\
&\cong E_*/\fm \otimes \T_n(E_*/\fm^k \otimes M)
\end{align*}
as both tensor products and $\T_n$ commute with filtered colimits. Here, the second isomorphism uses the claim for finite free modules proven above. Finally, any module can be written as a reflexive coequalizer of free modules, so the result follows. 
\end{proof}

We are now ready to prove our main result (\myref{mainthm}). 

\begin{thm}\mylabel{MainThmRestated}
The algebraic approximation functor $\T \colon \Mod_{E_*} \to \Mod_{E_*}$ preserves $L_0$-equivalences. 
\end{thm}
\begin{proof}
Since $\T = \bigoplus_{n \geq 0} \T_n$ is a coproduct of the functors $\T_n$, it suffices to show that each $\T_n$ preserves $L_0$-equivalences.
The statement then immediately translates into the following claim: The natural transformation $L_0 \T_n \xrightarrow{\simeq}L_0 \T_n\iota L_0$ is an isomorphism. Because $\T_n$, $L_0$, and $\iota$ all preserve reflexive coequalizers, we are reduced to showing this for free modules. 

Let $M$ be a free $E_*$-module and let $k$ be as in \myref{keygeneral}. We have a commutative diagram\footnote{Note that, by \myref{lprop}, $(L_0M)/\fm^k = M/\fm^k = L_0(M/\fm^k)$, so the notation $L_0M/\fm^k$ is unambiguous.}
\[ \xymatrix{M \ar[r] \ar[d] & L_0M \ar[d] \\
 M/\fm^k \ar@{.>}[r]^-{\simeq} & L_0M/\fm^k} \]
where the vertical arrows are the canonical quotient maps and the bottom one is the induced isomorphism.  Applying the functor $ E_*/{\fm} \otimes \T_n(-)$ to this square, we obtain
\[ \xymatrix{E_*/\fm \otimes \T_n (M) \ar[r] \ar[d] & E_*/\fm \otimes \T_n (L_0M) \ar[d] \\
E_*/\fm \otimes \T_n (M/{\fm}^k) \ar[r] & E_*/\fm \otimes \T_n (L_0M/{\fm}^k).} \]
The bottom horizontal map is still an isomorphism, while the two vertical arrows are so by \myref{keygeneral}. Therefore, the top horizontal map is an isomorphism as well. Since $\T_n L_0M$ is flat by \myref{flatchar} and \myref{tpres}, this furnishes the result by \myref{detection}.
\end{proof}

\subsection{Proof of the key property}\label{ProofKey} 

We are left with the verification of \myref{key}. To this end, we will use an observation from homotopy theory (\myref{inv}), which is the only non-algebraic input to the proof of our theorem. The notation is as before; in particular, all smash products are taken to be in $\Mod_E$ unless otherwise stated, and we will denote a morphism of spectra by the same symbol as the induced map on homotopy groups. We have fixed a non-negative integer $n$ and let $m \le n$.

Recall that any element $\nu \in \pi_i E$ gives rise to a multiplication by $\nu$ map, $\nu\colon  \Si^i M \to M$, for every $M \in \Mod_E$, defined as the composite
\[S^i \wedge M \xrightarrow{\nu \wedge 1} E \wedge M \xrightarrow{\mu_M} M, \]
which can be inverted by forming the sequential colimit
\[ \nu^{-1}M = \mathrm{colim}\ M \xrightarrow{\Si^{-i} \nu} \Si^{-i} M \xrightarrow{\Si^{-2i} \nu} \Si^{-2i} M \to \cdots. \]
To simplify notation, we will omit these suspensions from now on. 

\begin{lemma}\mylabel{inv}
Let $\nu \in \pi_*E$ and $m \ge 1$, and let $M$ be an $E$-module. If $\nu$ acts invertibly on $M$, then $\nu$ acts invertibly on $\pi_*\LD_m(M)$. This applies in particular to $M = \nu^{-1} \Si^q E$ for any $q \in \Z$.
\end{lemma}
\begin{proof}
To start with, we show that $\nu$ acts invertibly on $\pi_*\D_m(M)$. To see this, note that $\nu$ acts invertibly on $M \in \Mod_E$ if and only if the natural map $M = E \wedge M \to \nu^{-1}E \wedge M = \nu^{-1}M$ is an equivalence. Since smash products commute with homotopy colimits, the property of $\nu$ acting invertibly is therefore preserved under homotopy colimits. Moreover, the equivalence $(\nu^{-1}E)^{\wedge 2} = \nu^{-1}E$ implies $(\nu^{-1} X) \wedge (\nu^{-1} Y) = \nu^{-1} (X \wedge Y)$ for any $E$-modules $X$ and $Y$. Therefore, the property of $\nu$ acting invertibly is preserved under finite smash products. It follows that $\nu$ acts invertibly on the extended power $\D_m (M) = (M^{\wedge m})_{h\Si_m}$.

Now recall that \myref{algtotopss} and \myref{localhom} say that, for any $N \in \Mod_E$, there exists a strongly convergent spectral sequence of $E_*$-modules 
\[ E^2_{s,t} = L_s(\pi_*N)_t \Rightarrow \pi_{s+t}L_KN \]
with $E_2^{s,t}=0$ if $s>h$. Taking $N = \D_m(M)$, the previous step implies that the $E^2$-term of the spectral sequence consists of $E_*$-modules on which $\nu$ acts invertibly. Since this property  is closed under kernels, cokernels, and extensions, the claim follows.
\end{proof}

As a consequence, we can show that elements in the maximal ideal of $E_*$ have to act nilpotently on $\T_m(F)$ mod $\fm$ for any finite free $E_*$-module $F$.

\begin{cor}\mylabel{nilp}
Let $\nu \in \fm \subset E_*$, $m \ge 1$, and fix $F \in \Modff_{E_*}$, then there exists a positive integer $\lambda = \lambda(\nu,m)$ such that the map
\[\T_m(\nu) \colon \T_m(F) \to \T_m(F)\]
satisfies $\T_m(\nu)^{\lambda} \equiv 0 \mod \fm$.
\end{cor}
\begin{proof}
Write $F = \pi_* M$ for some finite free $E$-module $M$. Since $\nu^{-1}M = \mathrm{colim}\ (M \xrightarrow{\nu} M \xrightarrow{\nu} \cdots) $ and the functor $\T_m \circ \pi_*$ preserves filtered colimits and finite free modules by \myref{tpres}, both $\nu^{-1}M$ and $\T_m \pi_*(\nu^{-1}M)$ are flat by Lazard's theorem. Therefore, \myref{flatcomparison} gives an isomorphism $L_0 \T_m\pi_*(\nu^{-1}M) \cong \pi_*\LD_m(\nu^{-1}M)$. Putting this together with \myref{lprop}(2), we get that 
\[E_*/\fm \otimes \mathrm{colim}\ ( \T_m \pi_* M \xrightarrow{\T_m\pi_* \nu} \T_m \pi_* M \xrightarrow{\T_m\pi_* \nu} \cdots )\] 
is equivalent to 
\begin{eqnarray*} E_*/\fm \otimes \T_m \pi_* (\mathrm{colim}\ M \xrightarrow{\nu} M \xrightarrow{\nu} \cdots ) & = & E_*/\fm \otimes L_0 \T_m \pi_* (\nu^{-1}M) \\ 
& = & E_*/\fm \otimes \pi_*\LD_m (\nu^{-1}M) \\
& = & 0,
\end{eqnarray*}
where the last equality follows from \myref{inv}. Since $\T_m (\pi_*M)$ has finite rank by \myref{tpres}, $\T_m(\nu)$ is nilpotent mod $\fm$.
\end{proof}

\begin{rmk}
Note that the integer $\lambda=\lambda(\nu,m)$ in the above corollary depends on the module $F$.
\end{rmk}

We are now ready to prove \myref{key}, which we restate for convenience. 

\begin{propwo}
There exists a positive integer $k=k(n)$ such that for all $m \le n$ and all $q \in \Z$, the natural map
\[ E_*/\fm \otimes \T_m(\Si^q E_*) \xrightarrow{} E_*/\fm \otimes \T_m(\Si^q E_*/\fm^k) \]
is an isomorphism for all.
\end{propwo}

\begin{proof}
First we claim that it suffices to prove that, for fixed $q \in \Z$, there exists a positive integer $k = k(n,q)$ such that for all $m \le n$ the map $E_*/\fm \otimes \T_m(\Si^q E_*) \xrightarrow{} E_*/\fm \otimes \T_m(\Si^q E_*/\fm^k)$ is an isomorphism. Indeed, because $E$ is 2-periodic, taking $k(n) = \max{(k(n,0),k(n,1))}$ then works for all $q$. Since the argument below is the same for any $q$, we will assume $q=0$ from now on.

The statement is clear for $m=0$, so let us assume $m \ge 1$. With notation as in \myref{nilp}, taking the finite free $E_*$-module $F = E_*$, let $\lambda(\nu) = \mathrm{max}\{\lambda(\nu,1),\ldots,\lambda(\nu,n)\}$ and set $k = \sum_{i=0}^{h-1}\lambda(u_i)$, the $u_i$ being the standard generators of $\fm$. Let $r(k)= \binom{h+k-1}{k}$, the number of generators of $\fm^k$, and define $I_k=\{ g_1=u_0^k, g_2= u_0^{k-1}u_1,\ldots,g_{r(k)}=u_{h-1}^k \}$ to be the set of standard generators for this ideal, such that $g_i = \prod_{i=0}^{h-1} u_i^{l_i}$ where, by choice of $k$, $l_i \ge \lambda(u_i)$ for at least one $i$. 

We can write $E_*/\fm^k$ as a reflexive coequalizer
\[ E_* \oplus E_*^{\oplus r(k)} \doublerightarrow{(1,0,\ldots,0)}{(1,g_1,\ldots,g_{r(k)})}  E_* \to E_*/\fm^k \to 0 \]
where the upper arrow is projection onto the first factor. Note that the extra copy of $E_*$ is added only to make the coequalizer \emph{reflexive}. 
By naturality, we obtain a commutative diagram
\[ \xymatrix{\T_m(E_* \oplus E_*^{\oplus r(k)}) \ar[r]^-{\simeq} \ar[d]_{\T_m(1,g_1,\ldots,g_{r(k)})} & \underset{i_0+\ldots+i_{r(k)}=m}{\bigoplus} \T_{i_0}(E_*)  \otimes \T_{i_1}(E_*) \otimes \cdots \otimes \T_{i_{r(k)}}(E_*) \ar[d]^{\bigoplus \T_{i_0}(1) \otimes \T_{i_1}(g_1) \otimes \cdots \otimes \T_{i_{r(k)}}(g_{r(k)})}\\
\T_m(E_*) \ar[r]^= & \T_m(E_*) } \]
Observe that, by our choice of $k$ and \myref{nilp}, for each summand of the right vertical map we get
\[ \T_{i_0}(1) \otimes \T_{i_1}(g_1) \otimes \cdots \otimes \T_{i_{r(k)}}(g_{r(k)}) 
\underset{\fm}{\equiv}  
\begin{cases}  1 & \mbox{if } i_0=m,\ i_1=\ldots=i_{r(k)}=0 \\  0 & \mbox{otherwise.} 
\end{cases}
\]
On the one hand, the colimit of the diagram
\[ E_*/\fm \otimes \T_m(E_* \oplus E_*^{\oplus r(k)}) \doublerightarrow{E_*/\fm \otimes \T_m(1,0,\ldots,0)}{E_*/\fm \otimes \T_m(1,g_1,\ldots,g_{r(k)})}  E_*/\fm \otimes \T_m(E_*) \]
computes $E_*/\fm \otimes \T_m(E_*/\fm^k)$. On the other hand, by the argument just given, this diagram is equivalent to 
\[ E_*/\fm \otimes \underset{i_0 +\cdots+i_{r(k)}=m}{\bigoplus} \T_{i_0}(E_*) \otimes \T_{i_1}(E_*) \otimes \cdots \otimes \T_{i_{r(k)}}(E_*) \rightrightarrows E_*/\fm \otimes \T_m(E_*) \]
where on each summand both maps are either the identity, namely precisely if $i_0=m$, or both 0; hence its colimit is $E_*/\fm \otimes \T_m(E_*)$. Therefore, $E_*/\fm \otimes \T_m(E_*/\fm^k) \cong E_*/\fm \otimes \T_m(E_*)$.
The situation is summarized in the following diagram:
\[\resizebox{\textwidth}{!}{\xymatrix{E_*/\fm \otimes \T_m(E_* \oplus E_*^{\oplus r(k)}) \ar@<.75ex>[r] \ar@<-.75ex>[r] \ar[d]^{\simeq} & E_*/\fm \otimes \T_m(E_*) \ar[r] \ar[d]^{=} & E_*/\fm \otimes \T_m(E_*/\fm^k) \ar[r] \ar@{.>}[d]_{\therefore}^{\simeq} & 0\\
E_*/\fm \otimes \underset{i_0 +\cdots+i_{r(k)}=m}{\bigoplus} \T_{i_0}(E_*)  \otimes \cdots \otimes \T_{i_{r(k)}}(E_*) \ar@<.75ex>[r] \ar@<-.75ex>[r] & E_*/\fm \otimes \T_m(E_*) \ar[r] & E_*/\fm \otimes \T_m(E_*) \ar[r] & 0.}}
\]
For any integer $k' \geq k(n)$, the map $E_*/\fm \otimes \T_m(\Si^q E_*) \to E_*/\fm \otimes \T_m(\Si^q E_*/\fm^{k'})$ is also an isomorphism for all $m \leq n$ and $q \in \Z$. Indeed, it is an epimorphism since $\T_m$ and $E_*/\fm \otimes -$ preserve epimorphisms, and it is a monomorphism because of the factorization $M \to M/\fm^{k'} \to M/\fm^{k(n)}$ and the assumption on $k(n)$.
\end{proof}

\begin{rmk}\mylabel{ReflexCoeq}
The method used in the proof to turn a coequalizer into a reflexive coequalizer works in any category with finite coproducts. Thus, a functor between such categories which preserves reflexive coequalizers automatically preserves regular epimorphisms, i.e., maps which are the coequalizer of some parallel pair. In particular, the functor $\T_n \colon \Mod_{E_*} \to \Mod_{E_*}$ preserves regular epimorphisms, which in $\Mod_{E_*}$ are precisely the surjective maps. 
\end{rmk}

\section{Algebraic structures and completion}\label{compstructures}

In this section, we consider the compatibility of various algebraic structures with $L$-completion. After stating the general result in \S\ref{compmonad}, we specialize to the familiar context of rings and modules in \S\ref{comprings}. The example of interest for this paper are the categories of $\la$- and $\te$-rings introduced and discussed in \S\ref{larings} and $\S\ref{complarings}$, which are used to give an alternative proof of the height 1 case of \myref{mainthm} in Section \ref{height1}.

\subsection{Completed monads and their algebras}\label{compmonad}

Categories of algebras for a monad transform nicely under the application of the completion functor $L_0$, which will become relevant in our discussion of the completed algebraic approximation functors at height 1 in the next section. 
 
\begin{thm}\mylabel{CompletedMonad}
Let $\cC$ be a category and $\iota \colon \cC_0 \to \cC$ the inclusion of a reflective full subcategory, with reflector $L \colon \cC \to \cC_0$, and unit $\eta \colon 1 \to \iota L$.

\begin{enumerate}
\item \label{ColimL0Equiv} If a map of diagrams in $\cC$ is objectwise an $L$-equivalence, then the map induced on colimits (if they exist) is an $L$-equivalence.
\item \label{RetractL0Equiv} A retract of an $L$-equivalence is also an $L$-equivalence. 
\end{enumerate}
Let $\cD$ be another category and $\iota \colon \cD_0 \to \cD$ the inclusion of a reflective full subcategory, with reflector $L \colon \cD \to \cD_0$ (by abuse of notation).
\begin{enumerate}
\setcounter{enumi}{2}
\item \label{ColimPreserv} Let $F \colon \cC \to \cD$ be a functor obtained as a colimit of a diagram of functors $F_i \colon \cC \to \cD$. If each $F_i$ preserves $L$-equivalences, then so does $F$.
\item \label{RetractPreserv} If a functor $F \colon \cC \to \cD$ preserves $L$-equivalences, then any retract of $F$ also preserves $L$-equivalences. 
\end{enumerate}
Let $\M \colon \cC \to \cC$ be a monad which preserves $L$-equivalences. Then the following holds. 
\begin{enumerate}
\setcounter{enumi}{4}
\item \label{InheritedMonad} $\MM= L \M \iota \colon \cC_0 \to \cC_0$ naturally inherits a monad structure from that of $\M$.
\item Given an $\MM$-algebra $Y$ with structure map $\psi \colon \MM Y \to Y$, the composite $\M Y \ral{\eta} L \M Y \ral{\psi} Y$ naturally makes $Y$ into an $\M$-algebra. This defines a functor $\Alg_{\MM} \to \Alghat_{\M}$. Here $\Alg_{\MM}$ denotes the category of $\MM$-algebras, and $\Alghat_{\M}$ denotes the full subcategory of $\Alg_{\M}$ consisting of $\M$-algebras whose underlying object is in $\cC_0$.
\item \label{AlgCompletedMonad} Let $X$ be an $\M$-algebra. Then $LX$ is naturally an $\MM$-algebra, in a unique way that is compatible with the $\M$-algebra structure of $X$, in the sense that $\eta \colon X \to LX$ is a map of $\M$-algebras. This defines a functor $L \colon \Alg_{\M} \to \Alg_{\MM}$.
\item \label{CatAlgCompletedMonad} The functor $\Alg_{\MM} \ral{\cong} \Alghat_{\M}$ is an equivalence of categories, with inverse equivalence the composite
\[
\xymatrix{
\Alghat_{\M} \ar[r]^-{\io} & \Alg_{\M} \ar[r]^{L} & \Alg_{\MM}.
}
\]
\end{enumerate}
\end{thm}

\begin{proof}
Parts (1)-(4) are straightforward.

(5) Consider the category $\Fun(\cC,\cC)$ of endofunctors of $\cC$, which is monoidal with respect to composition of functors.  Let $\Fun^{we}(\cC,\cC)$ denote the full subcategory of $\Fun(\cC,\cC)$ consisting of functors that preserve $L$-equivalences; this is a monoidal subcategory.

One readily checks that the functor $L_{\sharp} \io^{\sharp} =\io^{\sharp} L_{\sharp} \colon \Fun^{we}(\cC,\cC) \to \Fun(\cC_0,\cC_0)$, which associates to an endofunctor $F \colon \cC \to \cC$ the endofunctor $L F \io \colon \cC_0 \to \cC_0$, is strong monoidal, with structure map
\[
(L F \eta)^{-1} G \io \colon (L F \io)(L G \io) \ral{\cong} L (FG) \io.
\]
Hence, a monad $\M$ on $\cC$, i.e., a monoid in $\Fun(\cC,\cC)$, is naturally sent to a monoid $L_{\sharp} \io^{\sharp} (\M) = L \M \iota$ in $\Fun(\cC_0,\cC_0)$.

(6) Straightforward.

(7) Let $\phy \colon \M X \to X$ be the structure map of $X$ as $\M$-algebra, then the composite
\[
\xymatrix{
L \M LX & L \M X \ar[l]_-{L \M \eta}^-{\cong} \ar[r]^-{L \phy} & L X \\
}
\]
naturally makes $LX$ into an $\MM$-algebra.

Now we argue the compatibility and uniqueness claims. We are looking for a map $\phy_1 \colon \M(LX) \to LX$ satisfying $\phy_1 \circ (\M \eta) = \eta \circ \phy$, i.e., making the square in
\[
\xymatrix{
\M X \ar[dd]_{\phy} \ar[dr]^{\eta} \ar[rr]^{\M \eta} & & \M (LX) \ar@{-->}[dd]^(0.7){\phy_1} \ar[dr]^{\eta} & \\
& L(\M X) \ar[rr]^(0.35){L \M \eta}_(0.35){\cong} \ar@{-->}[dr]_{\phy_3} & & L(\M LX) \ar@{-->}[dl]^{\phy_2} \\
X \ar[rr]_{\eta} & & LX & \\
}
\]
commute. Since $LX$ is $L$-complete, a map $\phy_1$ making the square commute corresponds to a map $\phy_2 \colon L (\M LX) \to LX$ making the outer pentagram commute.

By naturality of $\eta$, the upper parallelogram commutes, i.e., $\eta \circ (\M \eta) = (L \M \eta) \circ \eta$ holds. Moreover, $L \M \eta$ is an isomorphism by assumption. Therefore $\phy_2$ makes the outer pentagram commute if and only if the corresponding map $\phy_3 = \phy_2 \circ L \M \eta$ makes the lower left triangle commute.

Since $LX$ is $L$-complete, such maps $\phy_3$ correspond to maps $\M X \to LX$ that are equal to $\eta \circ \phy$. In other words, there exists a unique such $\phy_3$, and it is $\phy_3 = L \phy$.

(8) This is a straightforward consequence of parts (6) and (7). 
\end{proof}

\subsection{Completed rings and their modules}\label{comprings}

Now we specialize to the case of $\Mod_R$ with its reflective subcategory $\Modhat_R$ of $L$-complete modules. Here $R$ is a commutative Noetherian ring (not assumed to be local) with a chosen ideal $\fm \subset R$. All tensor products will be over $R$ unless otherwise noted.

First, note that $R$ itself might not be $L$-complete as an $R$-module. However, we will see that replacing $R$ by $L_0 R$ is harmless.

\begin{prop}\mylabel{L0Mod}
\begin{enumerate}
\item $L_0 R$ admits a unique $R$-algebra structure compatible with that of $R$, in the sense that $\eta \colon R \to L_0 R$ is a map of $R$-algebras. Moreover, said $R$-algebra structure on $L_0 R$ is natural in $R$, and is commutative.
\item Let $M$ be an $R$-module. Then $L_0 M$ admits a unique $L_0 R$-module structure such that its restriction along $\eta \colon R \to L_0 R$ is the original $R$-module structure on $L_0 M$. Moreover, said $L_0 R$-module structure on $L_0 M$ is natural in $R$ and $M$.
\item \label{L0Alg} Let $A$ be an $R$-algebra. Then $L_0 A$ admits a unique $L_0 R$-algebra structure compatible with the $R$-algebra structure of $A$. Moreover, said algebra structure is natural in $A$ and $R$. The algebra $L_0 A$ is commutative if $A$ is.
\end{enumerate}
\end{prop}

\begin{proof}
In all three cases, the induced structure maps arise from applying $L_0$ to the original structure maps, and using the fact that $L_0 \colon \Mod_{R} \to \Modhat_{R}$ is strong monoidal, i.e., the natural isomorphism
\[
L_0(\eta \ot \eta) \colon L_0(M \ot N) \ral{\cong} L_0(L_0 M \ot L_0 N) = L_0M \hat{\ot} L_0N.
\]
The compatibility claim follows from naturality of $\eta \colon M \to L_0 M$. The uniqueness claim is proved as in \myref{CompletedMonad} \eqref{AlgCompletedMonad}.

Let us add some details about each point.

(1) Since $R$ is a commutative monoid in $(\Mod_R,\ot)$, $L_0 R$ is naturally a commutative monoid in $(\Modhat_R,\hat{\ot})$. Since $L_0 R$ is $L$-complete, this structure is the same as that of a commutative $R$-algebra.

(2) Since $M$ is a module object over the monoid object $R$ in $(\Mod_R,\ot)$, $L_0 M$ is naturally a module object over the monoid object $L_0 R$ in $(\Modhat_R,\hat{\ot})$.

(3) Statement (2) defines a functor $L_0 \colon \Mod_R \to \Modhat_{L_0 R}$, which is still strong monoidal, since $\eta \colon R \to L_0 R$ induces a natural isomorphism
\[
L_0 (L_0 N \ot_R L_0 N') \ral{\cong} L_0 (L_0 N \ot_{L_0 R} L_0 N').
\]
Hence this functor sends monoids (resp. commutative monoids) in $(\Mod_{R},\ot)$ to monoids (resp. commutative monoids) in $(\Modhat_{L_0 R},\hat{\ot})$.
\end{proof}

In the proof  of \myref{L0Mod}, $\Modhat_{L_0 R}$ denotes the full subcategory of $\Mod_{L_0 R}$ consisting of $L_0 R$-modules which are $L$-complete when viewed as $R$-modules. We now check that the other possible interpretation of the notation is in fact the same.

\begin{lemma}\mylabel{CompleteGroundRing}
The diagram
\[\xymatrix{
\Mod_{L_0 R} \ar[d]_{\eta^*} \ar[r]^{\hat{L}_0} & \Mod_{L_0 R} \ar[d]^{\eta^*} \\
\Mod_{R} \ar[r]_{L_0} & \Mod_{R} \\
}\]
commutes (up to natural isomorphism), where the top horizontal functor $\hat{L}_0$ is $L$-completion with respect to the ideal $\hat{\fm} := \left( \eta(\fm) \right) \subseteq L_0 R$.

In particular, an $L_0 R$-module $M$ is $L$-complete with respect to the ideal $\hat{\fm} \subseteq L_0 R$ if and only if the underlying $R$-module $\eta^* M$ is $L$-complete with respect to the ideal $\fm \subseteq R$.
\end{lemma}

\begin{proof}
Since the restriction functor $\eta^* \colon \Mod_{L_0 R} \to \Mod_R$ is exact, the composite $\eta^* \hat{L}_0$ is the $0^{\text{th}}$ left derived functor of $\eta^* (-)^{\wedge}_{\hat{\fm}}$ and likewise, the composite $L_0 \eta^*$ is the $0^{\text{th}}$ left derived functor of $(-)^{\wedge}_{\fm} \eta^*$. It therefore suffices to show that $\eta^* (-)^{\wedge}_{\hat{\fm}} \cong (-)^{\wedge}_{\fm} \eta^*$ is a natural isomorphism.

Let $M$ be an $L_0 R$-module. By exactness of $\eta^*$, we have
\begin{align*}
\eta^* (M^{\wedge}_{\hat{\fm}}) &= \eta^* \left( \lim_k M / \hat{\fm}^k M \right) \\
&= \lim_k \eta^* \left( M / \hat{\fm}^k M \right) \\
&= \lim_k \, (\eta^* M) / (\eta^* (\hat{\fm}^k M) ) \\
(\eta^* M)^{\wedge}_{\fm} &= \lim_k \, (\eta^* M) / \fm^k (\eta^* M)
\end{align*}
and those two towers of $R$-modules are in fact equal, given the equality of $R$-submodules $\fm^k (\eta^* M) = \eta^* (\hat{\fm}^k M)$.
\end{proof}

\begin{prop}\mylabel{EquivCatMod}
\begin{enumerate}
\item The restriction functor $\eta^* \colon \Mod_{L_0 R} \to \Mod_{R}$ induces an equivalence of categories
\[
\eta^* \colon \Modhat_{L_0 R} \ral{\cong} \Modhat_{R}.
\]
\item \label{EquivCatAlg} The restriction functor $\eta^* \colon \Com_{L_0 R} \to \Com_{R}$ induces an equivalence of categories
\[
\eta^* \colon \Comhat_{L_0 R} \ral{\cong} \Comhat_{R}.
\]
\end{enumerate}
\end{prop}

\begin{proof}
(1) Consider the commutative diagram of adjoint pairs
\[
\xymatrix @R=3pc @C=3pc {
\Mod_{R} \ar@<-0.6ex>[d]_-{L_0} \ar@<0.6ex>[r]^{L_0 R \ot_R -} & \Mod_{L_0 R} \ar@<0.6ex>[l]^{\eta^*} \ar@<-0.6ex>[d]_{L_0} \\
\Modhat_{R} \ar@<-0.6ex>[u]_-{\io} \ar@<0.6ex>[r] & \Modhat_{L_0 R} \ar@<0.6ex>[l]^{\eta^*} \ar@<-0.6ex>[u]_{\io} 
}
\]
where the right adjoints commute (on the nose) and therefore the left adjoints commute (up to natural isomorphism). In other words, the functor $L_0 \colon \Mod_{R} \to \Modhat_{L_0 R}$ coming from \myref{L0Mod} is naturally isomorphic to the composite
\[
\xymatrix @C=3pc {
\Mod_{R} \ar[r]^-{L_0 R \ot_R -} & \Mod_{L_0 R} \ar[r]^-{L_0} & \Modhat_{L_0 R}. \\
}
\]
One readily checks that the composite
\[
\xymatrix{
\Modhat_{R} \ar[r]^{\io} & \Mod_{R} \ar[r]^{L_0} & \Modhat_{L_0 R} \\
}
\]
is an inverse equivalence to $\eta^*$.

(2) Similar proof, using \myref{L0Mod} \eqref{L0Alg}.
\end{proof}

\subsection{$\lambda$- and $\theta$-rings}\label{larings}

\begin{defn}\mylabel{defLambda}
\cite[\S 1.2]{lambda} A \Def{$\la$-ring} $R$ is a commutative ring equipped with operations $\la^i \colon R \to R$ for all $i \geq 0$, satisfying the following equations:
\begin{enumerate}
\item $\la^0(x) = 1$ for all $x \in R$.
\item $\la^1(x) = x$ for all $x \in R$.
\item $\la^i(1) = 0$ for all $i \geq 2$.
\item $\la^n(x+y) = \sum_{i+j=n} \la^i(x) \la^j(y)$ for all $x,y \in R$.
\item \label{LambdaProd} $\la^n(xy) = P_n(\la^1(x), \ldots, \la^n(x); \la^1(y), \ldots, \la^n(y))$ for all $x,y \in R$.
\item $\la^m(\la^n(x)) = P_{m,n}(\la^1(x), \ldots, \la^{mn}(x))$ for all $x \in R$.
\end{enumerate}
Here $P_n$ and $P_{m,n}$ are certain polynomials with integer coefficients in $2n$ and $mn$ variables respectively, defined using elementary symmetric polynomials as in \cite[\S 1.1, 1.2]{lambda}; the precise formulas will play no role in this paper. From the structure of $\la$-ring, one can construct for all $i \geq 1$ \Def{Adams operations} $\psi^i \colon R \to R$, which are $\la$-ring homomorphisms, natural in $R$; see \cite[\S 3.1]{lambda} for details.

Let $R$ be a $\la$-ring. An \Def{$R$-$\la$-algebra} $S$ is a commutative $R$-algebra which is also a $\la$-ring, such that the two structures are compatible. The compatibility condition can be equivalently defined as the $R$-algebra structure map $R \to S$ being a map of $\la$-rings, or by allowing $x$ and $y$ in the product formula \eqref{LambdaProd} to be either in $R$ or $S$.
\end{defn}

\begin{example}\mylabel{ZLambda}
The integers $\Z$, endowed with operations $\la^i(n) = \binom{n}{i}$, form a $\la$-ring. The Adams operations $\psi^i \colon \Z \to \Z$ are all trivial, that is, $\psi^i(n) = n$ holds for all $n \in \Z$ and $i \geq 1$ \cite[3.4]{lambda}.
\end{example}

\begin{example}
The complex $K$-theory $K^0(X)$ of a paracompact Hausdorff space $X$ is naturally a $\la$-ring, where the operation $\la^i$ is induced by the $i^{\text{th}}$ exterior power of vector bundles. In this case, the Adams operations $\psi^i$ are the classical operations introduced by Adams.

In particular, the $K$-theory of a point $K^0(\text{pt}) \cong \Z$ is the $\la$-ring described in \myref{ZLambda}.
\end{example}

\begin{note}
For a ground $\la$-ring $R$, let $\la \Alg_R$ denote the category of $R$-$\la$-algebras. In particular, $\la \Alg_{\Z}$ is the category of $\la$-rings.
\end{note}

\begin{defn}\mylabel{defTheta}
\cite[2.1]{bousfield} Let $p$ be a prime. A \Def{$\te^p$-ring} (or just $\te$-ring if the prime $p$ is understood) $R$ is a commutative ring together with a function $\te^p \colon R \to R$ (also denoted $\te$) satisfying:
\begin{enumerate}
\item \label{AddCond} $\te(x+y) = \te x + \te y - \sum_{i=1}^{p-1} \frac{1}{p} \binom{p}{i} x^i y^{p-i}$
\item \label{MultCond} $\te(xy) = (\te x) y^p + x^p (\te y) + p (\te x)(\te y)$
\item \label{CstCond} $\te(1) = 0.$
\end{enumerate}

Let $R$ be a $\te$-ring. An \Def{$R$-$\te$-algebra} $S$ is a commutative $R$-algebra which is also a $\te$-ring, such that the two structures are compatible.
\end{defn}

\begin{note}
Let $p$ be a prime. For a ground $\te$-ring $R$, let $\te \Alg_R$ denote the category of $R$-$\te$-algebras. In particular, $\te \Alg_{\Z}$ is the category of $\te$-rings.
\end{note}

For any prime $p$, any $\la$-ring has an underlying $\te^p$-ring structure \cite[2.5]{bousfield}, in which the equation $\psi^p(x) = x^p + p \te^p(x)$ holds for all elements $x$. In light of this, given a $\te^p$-ring $R$, one defines the associated Adams operation $\psi^p \colon R \to R$ by $\psi^p(x) := x^p + p \te^p(x)$.

The underlying $\te^p$-ring structure defines a forgetful functor $\la \Alg_{\Z} \to \te \Alg_{\Z}$, and more generally:

\begin{lemma}\mylabel{laforget}
For any ground $\la$-ring $R$, there is a forgetful functor $\la \Alg_R \to \te \Alg_R$.
\end{lemma}

\begin{example}\mylabel{thetaZ}
The underlying $\te^p$-ring of the $\la$-ring $\Z$ described in \myref{ZLambda} is equipped with the operation $\te^p(x) = \frac{x-x^p}{p}$.
\end{example}

For any $\la$-ring $R$, the forgetful functor $U^{\la}_R \colon \la \Alg_R \to \Mod_R$ from $R$-$\la$-algebras to $R$-modules has a left adjoint $F^{\la}_R \colon \Mod_R \to \la \Alg_R$. In fact, both forgetful steps, first $\la \Alg_R \to \Com_R$ to commutative $R$-algebras and then $\Com_R \to \Mod_R$ to $R$-modules, have left adjoints. Denote by $T^{\la}_R := U^{\la}_R F^{\la}_R \colon \Mod_R \to \Mod_R$ the associated forget-of-free $R$-$\la$-algebra monad, which naturally comes with a decomposition $T^{\la}_R = \bigoplus_{n = 0}^{\infty} T^{\la}_{R,n}$.

Likewise, for a $\te$-ring $R$, let $T^{\te}_R \colon \Mod_R \to \Mod_R$ denote the associated forget-of-free $R$-$\te$-algebra monad. 

By \cite[Chapters 2, 3]{algtheories}, filtered colimits and reflexive coequalizers in $\la \Alg_R$ are computed in $\Set$, or equivalently in the intermediate categories $\Mod_R$ or $\Alg_R$. This implies the following result. 

\begin{prop}\mylabel{forgetsifted}
The forgetful functor $U^{\la}_R \colon \la \Alg_R \to \Mod_R$ preserves sifted colimits, and so does the composite $T^{\la}_R = U^{\la}_R F^{\la}_R \colon \Mod_R \to \Mod_R$.
\end{prop}

\subsection{Completions and base change of $\la$-rings}\label{complarings}

\begin{prop}\mylabel{L0Lambda}
Let $R$ be a $\la$-ring.
\begin{enumerate}
\item $L_0 R$ is naturally a $\la$-ring, in a unique way that is compatible with the $\la$-ring structure of $R$, in the sense that $\eta \colon R \to L_0 R$ is a map of $\la$-rings.
\item Let $A$ be an $R$-$\la$-algebra. Then $L_0 A$ is naturally an $L_0 R$-$\la$-algebra, in a unique way that is compatible with the $R$-$\la$-algebra structure of $A$.
\end{enumerate}
\end{prop}

\begin{proof}
(1) Since the structure of $\la$-ring is monadic over $\Mod_{\Z}$, the statement follows from \myref{CompletedMonad} \eqref{AlgCompletedMonad} and \myref{TLambda}.

(2) Similar to \myref{L0Mod}.
\end{proof}

\begin{rmk}
\myref{L0Lambda} is closely related to \cite[3.62, 3.63]{lambda}. The latter considers metric completions of $\la$-rings. When restricted to the $p$-adic topology, the former statement is more general.
\end{rmk}

\begin{prop}\mylabel{EquivCatLambdaAlg}
Let $R$ be $\la$-ring. Then the restriction functor $\eta^* \colon \la \Alg_{L_0 R} \to \la \Alg_{R}$ induces an equivalence of categories
\[
\eta^* \colon \la \Alghat_{L_0 R} \to \la \Alghat_{R}.
\]
\end{prop}

\begin{proof}
Similar to \myref{EquivCatMod}.
\end{proof}

We want to obtain a result over the ground $\la$-ring $\Z_p$ starting from an analogous result over the ground $\la$-ring $\Z$. Let us discuss change of $\la$-rings more generally.

As for ordinary commutative algebras, if $R$ is any ground $\la$-ring, and $A$ and $B$ are $R$-$\la$-algebras, then their tensor product $A \ot_{R} B$ is naturally an $R$-$\la$-algebra, and this construction makes the tensor product $\ot_R$ into the coproduct of $R$-$\la$-algebras. Indeed, the forgetful functor from $R$-$\la$-algebras to $R$-algebras is comonadic \cite[2.26]{lambda} and thus creates colimits; this holds more generally for any plethory \cite[1.10]{pleth}. As a consequence, one can restrict or extend scalars just as for commutative rings. Let $f \colon R \to S$ be a map of $\la$-rings. Then extension of scalars $f_! = S \ot_R -$ is left adjoint to restriction of scalars $f^*$ at the level of modules, commutative algebras, and $\la$-algebras. More precisely, there is a diagram of adjoint pairs
\begin{equation} \label{Forget}
\xymatrix @R=3pc @C=3pc {
\la \Alg_R \ar@<0.6ex>[d] \ar@<0.6ex>[r]^-{S \ot_R -} & \la \Alg_S \ar@<0.6ex>[l]^-{f^*} \ar@<0.6ex>[d] \\
\Com_R \ar@<0.6ex>[u] \ar@<0.6ex>[d] \ar@<0.6ex>[r]^{S \ot_R -} & \Com_S \ar@<0.6ex>[l]^{f^*} \ar@<0.6ex>[u] \ar@<0.6ex>[d] \\
\Mod_R \ar@<0.6ex>[u] \ar@<0.6ex>[r]^{S \ot_R -} & \Mod_S \ar@<0.6ex>[l]^{f^*} \ar@<0.6ex>[u] \\
}
\end{equation}
where the downward maps are forgetful functors, and the upward maps are their left adjoint (``free'' functors). In each square, the right adjoints commute (on the nose), and thus the left adjoints commute (up to natural isomorphism).

Consider the associated free $\la$-algebra monads $T_R := T^{\la}_R \colon \Mod_R \to \Mod_R$ and $T_S := T^{\la}_S \colon \Mod_S \to \Mod_S$ respectively. The map $f \colon R \to S$ induces a natural comparison map
\[
T_f \colon T_R f^* \to f^* T_S
\]
of functors $\Mod_S \to \Mod_R$.

\begin{prop} \mylabel{ChangeLambda}
Let $R$ be a $\la$-ring and assume $T_R$ preserves $L_0$-equivalences. Let $f \colon R \to S$ be a map of $\la$-rings.
\begin{enumerate}
 \item If the natural transformation $T_f$ is an $L_0$-equivalence, then $T_S$ preserves $L_0$-equivalences.
 \item If the unit of the adjunction $f_! \dashv f^*$ is an $L_0$-equivalence, i.e., if for all $R$-module $M$ the natural map
\[
M = R \ot_R M \to S \ot_R M = f^*(S \ot_R M)
\]
is an $L_0$-equivalence, then $T_f$ is an $L_0$-equivalence, and therefore $T_S$ preserves $L_0$-equivalences.
\end{enumerate}
\end{prop}

\begin{proof}
(1) Let $M$ be an $S$-module. Applying the natural transformation $T_f$ to the map $\eta \colon M \to L_0 M$ and then applying $L_0$ yields a commutative diagram
\[
\xymatrix{
L_0 T_R (f^* M) \ar[d]_{L_0 T_f}^{\cong} \ar[r]^-{L_0 T_R f^* \eta} \ar@/^2.5pc/[rr]^{L_0 T_R \eta}_{\cong} & L_0 T_R (f^* L_0 M) \ar[d]^{L_0 T_f}_{\cong} \ar@{=}[r] & L_0 T_R (L_0 f^* M) \\
L_0 f^* T_S M \ar@{=}[d] \ar[r]_-{L_0 f^* T_S \eta} & L_0 f^* T_S (L_0 M) \ar@{=}[d] & \\
f^* L_0 T_S M \ar[r]_-{f^* L_0 T_S \eta}^-{\therefore \cong} & f^* L_0 T_S (L_0 M) & \\
}
\]
in $\Modhat_R$. The top map is an isomorphism, since $T_R$ preserves $L_0$-equivalences, and the downward maps are isomorphisms, since $T_f$ is an $L_0$-equivalence. Hence the bottom map $f^* L_0 T_S \eta$ is an isomorphism, and so is $L_0 T_S \eta$ (since the restriction functor $f^*$ reflects isomorphisms).

(2) First, let us show that $T_f$ is an $L_0$-equivalence when applied to extended $S$-modules, i.e., those of the form $N = S \ot_R M$ for some $R$-module $M$.

By the commutativity of left adjoints in the diagram \eqref{Forget}, for any $R$-module $M$, we have
\begin{align*}
T_S (S \ot_R M) &= U_S F_S (S \ot_R M) \\
&= U_S (S \ot_R F_R M) \\
&= S \ot_R U_R F_R M \\
&= S \ot_R T_R M.
\end{align*}
Let $N = S \ot_R M$ be an extended $S$-module. Applying $L_0$ to $T_f$ yields the natural map 
\[
L_0 T_R (f^* N) \to L_0 f^* T_S N
\]
whose right-hand side is
\begin{align*}
L_0 f^* T_S N &= L_0 f^* T_S (S \ot_R M) \\
&\cong L_0 f^* (S \ot_R T_R M) \\
&\cong L_0 T_R M
\end{align*}
since the unit map $T_R M \to f^* (S \ot_R T_R M)$ is an $L_0$-equivalence by assumption.

Consider the unit map  $M \to f^* (S \ot_R M) = f^* N$ and apply $L_0 T_R$ to obtain the isomorphism
\[
L_0 T_R M \ral{\cong} L_0 T_R (f^* N).
\]
One readily checks that this is the desired inverse of the map
\[
L_0 T_R (f^* N) \to L_0 f^* T_S N \cong L_0 T_R M.
\]
Now, note that free $S$-modules are extended from $R$-modules. Moreover, the functors $f^* \colon \Mod_S \to \Mod_R$ and $L_0 \colon \Mod_R \to \Modhat_R$ preserve all colimits, while the functors $T_R$ and $T_S$ preserve reflexive coequalizers and filtered colimits. Hence, the collection of $S$-modules for which $T_f$ is an $L_0$-equivalence is closed under reflexive coequalizers and filtered colimits. Since every $S$-module is a reflexive coequalizer of free $S$-modules, it follows that $T_f$ is an $L_0$-equivalence for all $S$-modules.
\end{proof}

\begin{cor}\mylabel{CompletionLambda}
Let $R$ be a $\la$-ring. If $T^{\la}_R$ preserves $L_0$-equivalences, then $T^{\la}_{L_0 R}$ also preserves $L_0$-equivalences.
\end{cor}

\begin{proof}
The coaugmentation $\eta \colon R \to L_0 R$ is a map of $\la$-rings, by \myref{L0Lambda}. By \myref{EquivCatMod}, the unit and counit of the adjunction $\eta_! \dashv \eta^*$
\begin{align*}
&M \to \eta^*(L_0 R \ot_R M) \\
&L_0 R \ot_R (\eta^* N) \to N
\end{align*}
are $L_0$-equivalences and thus \myref{ChangeLambda} applies.
\end{proof}

\section{Operations on $p$-complete $K$-theory}\label{opsonK}

The main goal of this section is to show that at height $1$, the monad $\T \colon \Mod_{E_*} \to \Mod_{E_*}$ is the free $\Z/2$-graded $\Z_p$-$\te$-algebra monad (\myref{Tht1identification}). This result is certainly known to experts, see e.g., \cite{hopkinslocal}, and follows from work of McClure \cite[IX]{hinfty}, as mentioned for instance in \cite[1.1]{congruence}. We claim no novelty in that respect. What we do provide is a detailed proof, based on a careful study of the power operation $Q$ introduced by McClure, and how this relates to work of Bousfield \cite{bousfield}. 

\subsection{The Milnor sequence for completed $E$-homology}

Recall that the periodicity theorem \cite{nilp2} allows the construction of generalized Moore spectra $M_I$ of type $h$ for certain sequences of positive integers $I = (i_0,\ldots,i_{h-1})$, which are finite spectra of type $h$ with 
\[BP_*(M_I) = BP_*/(p^{i_0},\ldots,v_{h-1}^{i_{h-1}}).\]
These generalize the ordinary Moore spectra $M(p^r)=\mathrm{cofib}(S^0 \xrightarrow{p^r} S^0)$. Moreover, by choosing appropriate cofinal sequences $I_j$ of length $n$, one obtains a tower of type $n$ generalized Moore spectra
\[\cdots \to M_{I_{j+1}} \to M_{I_j} \to  \cdots \]
with the property that 
\[L_KX = \lim_{I_j} M_{I_j} \wedge L_hX\]
for any spectrum $X$. In the following, we will fix such a tower and simplify notation by omitting the index $j$.

Let $X$ be spectrum. By smashing the tower $(M_I)_I$ with $E\wedge X$, the resulting tower gives rise to a Milnor sequence
\[\xymatrix{0 \ar[r] & \lim^1_IE_{*+1}(X \wedge M_I) \ar[r] & E^{\wedge}_*(X) \ar[r] & \lim_I E_*(X \wedge M_I) \ar[r] & 0,}\]
where $\cK_{h,*}(X) = \lim_I E_*(X \wedge M_I)$ is also known as the \Def{Morava module} of $X$, as introduced in \cite{picard}.
Some authors, e.g., \cite{ghmr}, use completed $E$-homology in the outer terms of the Milnor sequence; that this makes no difference is the content of the next result.  

\begin{lemma}
If $M_I$ is a type $h$ generalized Moore spectrum, then
\[L_{K}(X \wedge M_I) = L_n(X \wedge M_I)\]
holds for any spectrum $X$. In particular: $E^{\wedge}_*(X \wedge M_I) = E_*(X \wedge M_I)$.
\end{lemma}
\begin{proof}
Consider the chromatic pullback square for $X \wedge M_I$:
\[\xymatrix{L_n(X \wedge M_I) \ar[r] \ar[d] & L_{K}(X \wedge M_I) \ar[d] \\
L_{h-1}(X \wedge M_I) \ar[r] & L_{h-1}L_{K}(X \wedge M_I).}\]
Since $M_I$ is $E_{h-1}$-acyclic, the lower left corner is contractible and it suffices to show the same for $L_{K}(X \wedge M_I)$. To this end, note that by exactness every localization functor preserves smashing with a finite spectrum, thus $L_{K}(X \wedge M_I) = L_{K}(X) \wedge M_I$. This is clearly $E_{h-1}$-acyclic, so the claim follows.
\end{proof}

In fact, the $\lim^1$ term in the Milnor sequence vanishes in many cases, as has been observed for example in \cite{stricklandduality, ghmr}.

\begin{prop}\mylabel{lim1van}
The $\lim^1$ term in the Milnor sequence for the completed $E$-homology of a spectrum $X$ vanishes, i.e., $\lim^1_IE_{*+1}(X \wedge M_I) = 0$, if $X$ satisfies one of the following conditions.
\begin{enumerate}
 \item $E^{\wedge}_*(X)$ is pro-free. In particular, this holds if $K_*(X)$ is evenly concentrated.   
 \item $X$ is dualizable in the $K$-local category. 
\end{enumerate}
\end{prop}
\begin{proof}
Since the references quoted above do not provide a proof, we sketch the argument, based on \cite{moravak}. For (1), if $E^{\wedge}_*(X)$ is pro-free, the sequence $(p,\ldots,u_{h-1})$ acts regularly on it, and so does $(p^{i_0},\ldots,u_{h-1}^{i_{h-1}})$ for any $I=(i_0,\ldots,i_{h-1})$. Consequently, we get
\[E^{\wedge}_*(X \wedge M_I) = E^{\wedge}_*(X)/(p^{i_0},\ldots,u_{h-1}^{i_{h-1}}).\]
This implies that the tower $E_{*}(X \wedge M_I)$ satisfies the Mittag-Leffler condition, so $\lim^1_IE_{*+1}(X \wedge M_I) = 0$. Furthermore, if $K_*(X)$ is concentrated in even degrees, \cite[8.4(f)]{moravak} shows that $E^{\wedge}_*(X)$ is pro-free.

In case (2), note that $E^{\wedge}_*(X \wedge M_I)$ is finite as an abelian group for all indices $I$, by \cite[8.6]{moravak} and an inductive argument. Hence the tower is Mittag-Leffler and the claim follows. 
\end{proof} 

\begin{rmk}
As Strickland remarks at the end of \cite{stricklandduality}, for $X = \bigvee_I M_I$ we get $\lim^1_I E_{*+1}(X \wedge M_I) \ne 0$. Taking $\D X$ and considering the degree 1 part gives an $\Hinf_{\infty}$-ring spectrum counterexample. 
\end{rmk}

\subsection{The operation $Q$}

We now specialize to height $1$ for the rest of the section. Recall that at height $h=1$, Morava $E$-theory is $p$-completed complex $K$-theory: $E_1 = K^{\wedge}_p$. Its coefficient ring is $E_* = \Z_p[u^{\pm}]$ with $\abs{u} = 2$, where $\Z_p$ denotes the $p$-adic integers. The maximal ideal is $\fm = (p) \subset E_*$. Completed $E$-homology at height $1$ is given by
\begin{align*}
E^{\wedge}_*(X) &= \pi_* L_{K(1)} (K^{\wedge}_p \wedge X) \\
&= \pi_* \left( K^{\wedge}_p \wedge X \right)^{\wedge}_p \\
&= \pi_* \left( K \wedge X \right)^{\wedge}_p.
\end{align*}
We will compare Rezk's construction of $\T$ to the work of McClure, a seminal reference for explicit computations at height $1$. To use his results, we start with his operation $Q$ on the $K$-homology (with coefficients) of $\HH$-ring spectra, and turn it into an operation on the homotopy of $p$-complete $\HH$ $K$-algebras (\myref{OpQ}).

Let $Y$ be an $\Hinf_{\infty}$-ring spectrum. We use the notation $K_r := K \wedge M(p^r)$ and $K_*(Y;\Z/p^{r}) = \pi_*(K_r \wedge Y)$. 
The operations $Q \colon K_*(Y; \Z/p^r) \to K_*(Y; \Z/p^{r-1})$ described in \cite[IX 3.3]{hinfty} are compatible with the natural projection maps, i.e., the following diagram commutes
\[\xymatrix{K_*(Y;\Z/p^{r+1}) \ar[r]^-Q \ar[d]_{\pi} & K(Y;\Z/p^{r}) \ar[d]^{\pi} \\
K_*(Y;\Z/p^{r}) \ar[r]_-Q & K_*(Y;\Z/p^{r-1});}\]
hence they induce natural operations on inverse limits and $\lim^1$ terms as illustrated in the diagram
\begin{equation} \label{MilnorSeq}
\xymatrix{
0 \ar[r] & \lim^{1}_r K_{*+1} (Y;\Z/p^r) \ar[d]_{Q} \ar[r] & K_*^{\wedge}(Y) \ar@{-->}[d]_{?} \ar[r] & \lim_r K_* (Y; \Z/p^r) \ar[d]_{Q} \ar[r] & 0 \\
0 \ar[r] & \lim^{1}_r K_{*+1} (Y;\Z/p^r) \ar[r] & K_*^{\wedge}(Y) \ar[r] & \lim_r K_* (Y; \Z/p^r) \ar[r] & 0 \\
}
\end{equation}
where each row is the Milnor sequence for $Y$. We will show below that $Q$ induces a natural operation $Q \colon K_*^{\wedge}(Y) \to K_*^{\wedge}(Y)$ making the right-hand square of the diagram \eqref{MilnorSeq} commute. Said operation will be constructed using universal examples, namely, extended powers of spheres.

As a warm-up, assume that $Y$ is an $\Hinf_{\infty}$-ring spectrum satisfying one of the conditions of \myref{lim1van}. The Milnor sequence for the $p$-complete $K$-theory of $Y$ then simplifies to yield a commutative square 
\[\xymatrix{K^{\wedge}_*(Y) \ar[r]^-{\simeq} \ar@{-->}[d]_Q & \lim_rK_*(Y;\Z/p^{r}) \ar[d]^Q \\
K^{\wedge}_*(Y) \ar[r]^-{\simeq} & \lim_rK_*(Y;\Z/p^{r}),}\] 
where the left vertical arrow is the induced morphism. 

\begin{cor}\mylabel{qeasy}
For $Y$ as above, there exists a unique operation $Q\colon K^{\wedge}_*(Y) \to K^{\wedge}_*(Y)$ compatible with McClure's operation $Q$.
\end{cor}

To treat the case of an arbitrary $\mathbb{H}_{\infty}$-ring spectrum $Y$, we require the following result, whose proof was provided by Jim McClure. By abuse of notation, we use the same name for a class and its image under the natural map $K_*^{\wedge}(Y) \to K_*(Y; \Z/p^r)$.

\begin{prop}\mylabel{claim}
Let $j = 0$ or $1$, and let $x \in K_j^{\wedge} \left( \D^S(S^j) \right)$ be the fundamental class, i.e., the image of the generator $x \in K_j^{\wedge} \left( S^j \right) = \Z_p$ under the unit map $S^j = \D_1^S(S^j) \to \D^S(S^j)$. For any integer $r \geq 1$, $K_*(\D^S(S^0); \Z/p^r)$ is the free strictly\footnote{A graded ring $R$ is called \Def{strictly commutative} if $x^2 = 0$ holds for all $x$ of odd degree, which is not automatic if $R$ has $2$-torsion.} commutative $\Z/2$-graded algebra over $\Z/p^r$ on generators $\{ Q^i x \mid i \geq 0 \}$. 

Explicitly: For an even degree generator, we obtain the $\Z/2$-graded polynomial algebra
\[
K_* \left( \D^S(S^0); \Z/p^r \right) = \Z/p^r [x, Q x, Q^2 x, \ldots]
\]
while for an odd degree generator, we obtain the $\Z/2$-graded exterior algebra
\[
K_* \left( \D^S(S^1); \Z/p^r \right) = \La_{\Z/p^r} [x, Q x, Q^2 x, \ldots].
\]
Note that the elements $Q^i x$ are well defined here, since $x \in K_j \left( \D^S(S^j); \Z/p^r \right)$ is defined as the image of $x \in \lim_r K_j \left( \D^S(S^j); \Z/p^r \right)$.
\end{prop}

\begin{proof}
First consider $j=0$. In the notation of \cite[IX \S 2]{hinfty}, our $\D^S(S^j)$ corresponds to $C(S^0 \vee S^j)$, and \cite[IX 3.10]{hinfty} provides the desired isomorphism
\[
K_* \left( \D^S(S^0); \Z/p \right) = \Z/p [x, Q x, Q^2 x, \ldots]
\]
for $r=1$. The $K$-homology Bockstein spectral sequence for $\D^S(S^0)$, with $E^1$ term
\[
E^1_s = K_s \left( \D^S(S^0); \Z/p \right)
\]
and differentials $d^r \colon E^r_s \to E^r_{s-1}$ collapses at $E^1$, since the $E^1$ term is concentrated in even degrees. Therefore each $K_* \left( \D^S(S^0); \Z/p^r \right)$ is a free $\Z/p^r$-module. Moreover, the generators of $\Z/p^r [x, Q x, Q^2 x, \ldots]$ map to a basis of the $E^r$ term, which implies the desired isomorphism
\[
K_* \left( \D^S(S^0); \Z/p^r \right) = \Z/p^r [x, Q x, Q^2 x, \ldots]
\]
of algebras.

The proof is similar for $\D^S(S^1)$, again using \cite[IX 3.3]{hinfty}. The collapsing of the Bockstein spectral sequence in that case follows from part (v). The fact that $K_* \left( \D^S(S^1); \Z/p^r \right)$ is \emph{strictly} graded commutative follows from part (x).
\end{proof}

\begin{cor}\mylabel{kps}
There is an isomorphism 
\[K_*^{\wedge} \left( \D^S(S^0) \right) = L_0 \left( \Z_p [x, Q x, Q^2x, \ldots] \right)\] and similarly for a generator in odd degree. 
\end{cor}
\begin{proof}
We have
\[
\left( K \wedge \D^S(S^0) \right)_p^{\wedge} = \holim_r \left( K_r \wedge \D^S(S^0) \right)
\]
and the homotopy of each stage of this tower is $K$-homology with coefficients $\pi_* \left( K_r \wedge \D^S(S^0) \right) = K_* \left( \D^S(S^0); \Z/p^r \right)$. We now invoke \myref{claim} to prove the corollary. Indeed, the tower:
\[
\{ K_* \left( \D^S(S^0); \Z/p^r \right) \}_r
\]
satisfies the Mittag-Leffler condition, so that the $\lim^{1}$ term in the Milnor exact sequence
vanishes, yielding the isomorphism
\begin{align*}
K_*^{\wedge} \left( \D^S(S^0) \right) &= \lim_r K_* \left( \D^S(S^0); \Z/p^r \right) \\
&= \lim_r \left( \Z/p^r [x, Q x, Q^2 x, \ldots] \right) \\
&= \Z_p [x, Q x, Q^2 x,\ldots]_p^{\wedge} \\
&= L_0 \left( \Z_p [x, Q x, Q^2 x, \ldots] \right).
\end{align*}
which gives the claim. 
\end{proof}


We now prove the claims made at the beginning of this section.

\begin{lemma}
Let $\phy \colon \left( K \wedge \D^S(S^j) \right)_p^{\wedge} \to \left( K \wedge \D^S(S^k) \right)_p^{\wedge}$ be an $\HH$ $K$-algebra map. Then $\phy$ is compatible with the operation $Q$, in the sense that the diagram
\[
\xymatrix{
K_i^{\wedge} \left( \D^S(S^j) \right) \ar[d]_Q \ar[r]^{\pi_i \phy} & K_i^{\wedge} \left( \D^S(S^k) \right) \ar[d]_Q \\
K_i^{\wedge} \left( \D^S(S^j) \right) \ar[r]^{\pi_i \phy} & K_i^{\wedge} \left( \D^S(S^k) \right) \\
}
\]
commutes.
\end{lemma}
\begin{proof}
Let $F$ be an $\HH$-ring spectrum and $e \in F_m (\D^S_q(S^n))$. Consider the (internal) Dyer-Lashof operation $Q_e \colon F_n Y \to F_m Y$ on the $F$-homology of $\HH$ ring spectra $Y$ constructed in \cite[IX 1.1]{hinfty}. It follows from the construction that $Q_e$ is natural with respect to $\HH$ $F$-algebra maps.

Next the construction of these operations readily adapts to completed homology. This follows from the fact that smash products and extended powers preserve homology isomorphisms for any homology theory. To conclude, note that the operation $Q \colon K_i^{\wedge} \left( \D^S(S^k) \right) \to K_i^{\wedge} \left( \D^S(S^k) \right)$ is an instance of this construction, taking $F = K_p^{\wedge}$, $n=m=i$, $q=p$, and the class $e = Qx \in K_i^{\wedge} \left( \D^S_p(S^i) \right)$.
\end{proof}

\begin{prop}\mylabel{OpQ}
There is a unique operation $Q \colon \pi_* A \to \pi_* A$ acting on the homotopy of any $p$-complete $\HH$ $K$-algebra $A$ such that $Q$ is natural with respect to $\HH$ $K$-algebra maps and agrees with the operation
\[
Q \colon K_*^{\wedge} \left( \D^S(S^k) \right) \to K_*^{\wedge} \left( \D^S(S^k) \right) 
\]
from the proof of \myref{kps} in the case $A = \left( K \wedge \D^S(S^k) \right)_p^{\wedge}$.

In particular, $Q$ induces an operation $Q \colon K_*^{\wedge}(Y) \to K_*^{\wedge}(Y)$ on the completed $K$-homology of any $\HH$-ring spectrum $Y$, natural with respect to $\HH$ $K$-algebra maps $\left( K \wedge Y \right)_p^{\wedge} \to \left( K \wedge Y' \right)_p^{\wedge}$.
\end{prop}

\begin{proof}
Note that a $p$-complete $K$-algebra is the same as a $p$-complete $K_p^{\wedge}$-algebra. The result follows from the fact that $\left( K \wedge \D^S(S^j) \right)_p^{\wedge}$ is the \emph{free} $p$-complete $\HH$ $K$-algebra on one generator in degree $j$.

Explicitly: Let $A$ be a $p$-complete $\HH$ $K$-algebra and $a \in \pi_j A$ be represented by a map $f \colon S^j \to A$. Then there is a unique (up to homotopy) $\HH$ $K$-algebra map $\tilde{f}$ making the diagram
\[
\xymatrix{
S^j \ar[d] \ar[r]^f & A \\
\D^S(S^j) \ar[d]_{\eta_K \wedge 1} \ar@{-->}[ur]& \\
**[l] \D^K(\Si^j K) = K \wedge \D^S(S^j) \ar[d] \ar@{-->}[uur] & \\
\left( K \wedge \D^S(S^j) \right)_p^{\wedge} \ar@{-->}[uuur]_{\tilde{f}} & \\
}
\]
commute (up to homotopy). Here $\eta_K \colon S^0 \to K$ denotes the unit map. By construction, the map induced on homotopy
\[
\pi_*(\tilde{f}) \colon \pi_* \left( K \wedge \D^S(S^j) \right)_p^{\wedge} \to \pi_* A
\]
sends the fundamental class $x$ (represented by the downward composite on the left) to $a \in \pi_j A$. Setting $Q(a) = Q \left( \pi_*(\tilde{f})(x) \right) := \pi_*(\tilde{f}) (Qx)$ constructs the desired operation $Q$, which moreover is determined by its effect on the universal examples $\pi_* \left( K \wedge \D^S(S^j) \right)_p^{\wedge}$.

This operation agrees with the original $Q$ when $A$ is of the form $\left( K \wedge \D^S(S^k) \right)_p^{\wedge}$, by naturality of the original $Q$ with respect to $\HH$ $K$-algebra maps.
\end{proof}

\begin{rmk}
Hopkins \cite{hopkinslocal} gives a direct construction of the operation $Q \colon \pi_* A \to \pi_* A$ for any $K(1)$-local $\Hinf_{\infty}$-ring spectrum $A$, using the $K(1)$-local splitting $\Sigma_+^{\infty}B\Sigma_p = S^0 \vee S^0$ and the total power operation. By \myref{OpQ}, his operation coincides with ours by uniqueness. 
\end{rmk}

\begin{prop}\mylabel{Qlim0}
The operation $Q \colon K_*^{\wedge}(Y) \to K_*^{\wedge}(Y)$ on the completed $K$-homology of $\HH$-ring spectra $Y$ constructed in \myref{OpQ} is compatible with reduction of coefficients, in the sense that the diagram on the left 
\[
\xymatrix{
K_*^{\wedge}(Y) \ar[d] \ar[r]^Q & K_*^{\wedge}(Y) \ar[d] & K_*^{\wedge}(Y) \ar[d] \ar[r]^Q & K_*^{\wedge}(Y) \ar[d]\\
K_*(Y; \Z/p^{r}) \ar[r]^Q & K_*(Y; \Z/p^{r-1}) & \lim_r K_*(Y; \Z/p^{r}) \ar[r]^Q & \lim_r K_*(Y; \Z/p^{r}) \\
}
\]
commutes for all $r \geq 2$. Consequently, the diagram on the right also commutes. 
\end{prop}

\begin{proof}
Let $y \in K_j^{\wedge}(Y)$ and let $\tilde{f} \colon \left( K \wedge \D^S(S^j) \right)_p^{\wedge} \to \left( K \wedge Y \right)_p^{\wedge}$ be an $\HH$ $K$-algebra map as in the proof of \myref{OpQ}, satisfying $\pi_j(\tilde{f})(x) = y$, where $x \in K_j^{\wedge} \left( \D^S(S^j) \right)$ is the fundamental class. Consider the cube:
\[\resizebox{\columnwidth}{!}{
\xymatrix{
& K_j^{\wedge}(Y) \ar[dd] \ar[rr]^(0.4){Q} & & K_j^{\wedge}(Y) \ar[dd] \\
K_j^{\wedge} \left( \D^S(S^j) \right) \ar[dd] \ar[rr]^(0.6){Q} \ar[ur]^{\tilde{f}_*} & & K_j^{\wedge} \left( \D^S(S^j) \right) \ar[dd] \ar[ur]^{\tilde{f}_*}& \\
& K_j(Y; \Z/p^{r}) \ar[rr]^(0.4){Q} & & K_j(Y; \Z/p^{r-1}) \\
K_j \left( \D^S(S^j); \Z/p^{r} \right) \ar[rr]^(0.6){Q} \ar[ur]^{\tilde{f}_*}& & K_j\left( \D^S(S^j); \Z/p^{r-1} \right) \ar[ur]^{\tilde{f}_*}. & \\
}}
\]
The left and right faces commute, by naturality of reduction of coefficients, induced by the natural map $X_p^{\wedge} \to X/p^r$ for any spectrum $X$. The top and bottom faces commute, by naturality of $Q$ with respect to $\HH$ $K$-algebra maps. The front face commutes, because of the natural isomorphism $K_j^{\wedge} \left( \D^S(S^j) \right) = \lim_r K_j \left( \D^S(S^j); \Z/p^{r} \right)$, in other words, the $\lim^1$ term vanishes in this case. Therefore the back face commutes when restricted to the image of $\pi_j(\tilde{f})$.
\end{proof}

\begin{rmk}
It seems likely, though not obvious, that the operation $Q \colon K_*^{\wedge}(Y) \to K_*^{\wedge}(Y)$ constructed in \myref{OpQ} would also be compatible with the $\lim^1$ term of the Milnor sequence, i.e., make the left-hand square of the diagram \eqref{MilnorSeq} commute. We leave this as a question to the reader.
\end{rmk}

\subsection{$K$-theory of $\Hinf_{\infty}$-ring spectra}\label{Hinf}

Applying the results of the previous section, we now turn to the description of the structure found on the completed $K$-theory of any $\Hinf_{\infty}$-ring spectrum $Y$.

\begin{defn}
\cite[2.3]{bousfield} Let $p$ be a prime. A \Def{$\Z/2$-graded $\te$-ring} is a strictly commutative $\Z/2$-graded ring $R$ equipped with an operation $\te \colon R_0 \to R_0$ making $R_0$ into a $\te$-ring, and a group homomorphism $\psi \colon R_1 \to R_1$, such that the following conditions hold:
\begin{enumerate}
\item $\psi(ax) = \psi(a) \psi(x)$ for all $a \in R_0$ and $x \in R_1$.
\item $\te(xy) = \psi(x) \psi(y)$ for all $x, y \in R_1$.
\end{enumerate}
Here, $\psi \colon R_0 \to R_0$ denotes the associated Adams operation $\psi(a) = a^p + p \te(a)$.
\end{defn}

\begin{lemma}
Let $Y$ be an $\Hinf_{\infty}$-ring spectrum. Then the operations
\begin{align*}
&\te := Q \colon K_0^{\wedge}(Y) \to K_0^{\wedge}(Y) \\
&\psi := Q \colon K_1^{\wedge}(Y) \to K_1^{\wedge}(Y)
\end{align*}
make $K_*^{\wedge}(Y)$ into a $\Z/2$-graded $\te$-ring, in fact a $\Z/2$-graded $\Z_p$-$\te$-algebra. \end{lemma}

\begin{proof}
The statement follows from \cite[IX 3.3]{hinfty}; we will refer to its parts numbered in lowercase Roman numerals (i)-(x).

We first check that $\te$ makes $K_0^{\wedge}(Y)$ into a $\te$-ring. That $\te$ satisfies the three conditions in \myref{defTheta} follows from parts (vi), (vii), and (ii), respectively. Note that this holds also for $p=2$.

$K_*^{\wedge}(Y)$ is a commutative $\Z/2$-graded ring, and strict commutativity follows from part (x). By part (vi), $\psi \colon K_1^{\wedge}(Y) \to K_1^{\wedge}(Y)$ is a group homomorphism.

For $a \in K_0^{\wedge}(Y)$ and $x \in K_1^{\wedge}(Y)$, part (vii) yields
\begin{align*}
\psi(ax) &= a^p \psi(x) + p \te(a) \psi(x) \\
&= \left( a^p + p \te(a) \right) \psi(x) \\
&= \psi(a) \psi(x).
\end{align*}

For $x,y \in K_1^{\wedge}(Y)$, the condition $\te(xy) = \psi(x) \psi(y)$ also follows from part (vii).

The structure map of $K_0^{\wedge}(Y)$ as $\Z_p$-algebra is
\[
\xymatrix{
\Z_p = K_0^{\wedge}(S^0) \ar[r]^-{e_*} & K_0^{\wedge}(Y) \\
}
\]
where $e \colon S^0 \to Y$ denotes the unit map of $Y$, which is an $\Hinf_{\infty}$ map \cite[I 3.4(i)]{hinfty}. By part (i), $Q$ is natural with respect to $\Hinf_{\infty}$ maps and thus $e_*$ commutes with $Q$. It remains to check that $Q$ induces the usual $\te$-ring structure on $K_0^{\wedge}(S^0) = \Z_p$. Let $u \in K_0^{\wedge}(S^0)$ be the unit. For any $k \in \Z$,  parts (vi) and (ii) yield:
\begin{align*}
Q(ku) &= k Q(u) - \frac{1}{p} (k^p - k) u^p \\
&= - \frac{1}{p} (k^p - k) u \\
&= \frac{k - k^p}{p} u
\end{align*}
as in \myref{thetaZ}.
\end{proof}

Finally, we are ready to identify the monad $\T$ at height $1$, where Morava $E$-theory is $E = K^{\wedge}_p$.

\begin{thm}\mylabel{Tht1identification}
At height $h=1$, the monad $\T \colon \Mod_{E_*} \to \Mod_{E_*}$ is the free $\Z/2$-graded $\te$-ring over the ground $\te$-ring $\Z_p$, as defined in \cite[2.3]{bousfield}. In particular, the degree $0$ part of $\T$ is the free $\Z_p$-$\te$-algebra monad $T^{\te}_{\Z_p} \colon \Mod_{\Z_p} \to \Mod_{\Z_p}$. 
\end{thm}

\begin{proof}
Let $gT^{\te}_{\Z_p} \colon \Mod_{E_*} \to \Mod_{E_*}$ denote the free $\Z/2$-graded $\Z_p$-$\te$-algebra monad. Since both $\T$ and $gT^{\te}_{\Z_p}$ preserve filtered colimits and reflexive coequalizers, it suffices to show that they agree on the subcategory $\Modff_{E_*}$ of finite free $E_*$-modules. Moreover, both functors send finite direct sums to tensor products, hence it suffices to show that they agree on $E_*$ and $\Si E_*$, the free $E_*$-modules on one generator, in a way that is compatible with the respective unit maps of the two monads $1 \to gT^{\te}_{\Z_p}$ and $1 \to \T$.

By \cite[2.6, 3.1]{bousfield}, the free $\Z/2$-graded $\Z_p$-$\te$-algebra on one generator $x$ in degree $0$ is the $\Z/2$-graded polynomial ring
\[
gT^{\te}_{\Z_p}(E_*) = \Z_p [x, \te x, \te^2 x, \ldots],
\]
whereas on one generator $y$ in degree $1$, we get the $\Z/2$-graded exterior algebra
\[
gT^{\te}_{\Z_p}(\Si E_*) = \Lambda_{\Z_p} [y, \psi^p y, (\psi^p)^2 y, \ldots].
\]
We omit the $\Z/2$-grading from the notation.

Let us focus on the case of a generator in degree $0$. Starting from the definition of the functor $\T$, one obtains:
\begin{align*}
\T(E_*) &= \bigoplus_{n \geq 0} \pi_* L_{K(1)} \D^E_n(E) \\
&= \bigoplus_{n \geq 0} \pi_* L_{K(1)} (E \wedge_S B \Si_{n+}) \\
&= \bigoplus_{n \geq 0} \pi_* (K \wedge_S B \Si_{n+})_p^{\wedge}
\end{align*}
and similarly for $T(\Si E_*)$.

Note that although $\T_n(E_*) \cong \pi_* L_K \D_n(E)$ holds for each $n$, $\T(E_*)$ becomes $\pi_* L_K \D(E)$ only after applying $L$-completion:
\[
\pi_* L_K \D(E) \cong L_0 \T(E_*) = L_0 \left( \bigoplus_{n \geq 0} \T_n(E_*) \right),
\]
in which $\fm$-adically decaying infinite sums are allowed.

Each $\T_n(E_*)$ picks up the elements of weight $n$, where $x$ has weight $1$ and $Q^i x$ has weight $p^i$. In particular, $\T_n(E_*)$ is a finitely generated free $\Z_p$-module. For example:
\begin{align*}
&\T_0(E_*) = \Z_p \lan 1 \ran \\
&\T_1(E_*) = \Z_p \lan x \ran \\
&\ldots \\
&\T_{p-1}(E_*) = \Z_p \lan x^{p-1} \ran \\
&\T_p(E_*) = \Z_p \lan x^p, Qx \ran.
\end{align*}
Therefore, $\T(E_*)$ is the uncompleted $\Z/2$-graded polynomial algebra
\[
\T(E_*) = \Z_p [x, Q x, Q^2 x, \ldots],
\]
using \myref{kps}. The case of $\T (\Si E_*)$ is proved similarly.
\end{proof}

\begin{rmk}
At the prime $p=2$, the statement is also found in \cite[2.1]{spin}, where the free $\HH$-ring spectrum on one generator $\D^S(S^0) = \bigvee_{n \geq 0} B \Si_{n+}$ is denoted $T S^0$.
\end{rmk}

\section{The height 1 case}\label{height1}

The main result of this section is that the monad $\T$, identified as the free $\Z/2$-graded $\Z_p$-$\te$-algebra monad in \myref{Tht1identification}, preserves $L_0$-equivalences. Since the $2$-periodic grading does not play a crucial role here, we will focus on the degree $0$ part, which coincides with the free $\Z_p$-$\te$-algebra monad $T^{\te}_{\Z_p} \colon \Mod_{\Z_p} \to \Mod_{\Z_p}$.

While this can be viewed as the height $1$ case of \myref{mainthm}, the methods in this section make the structure present on the completed algebraic approximation functors explicit for $p$-complete $K$-theory.

We first prove the analogous result for $\la$-rings over $\Z$ (\myref{TLambda}), using the combinatorics of $\la$-rings. Then we obtain the result for $\la$-rings over $\Z_p$ (\myref{TLambdaZp}) by a change of $\la$-rings argument (\myref{CompletionLambda}). From there we deduce the result for $\te$-rings over $\Z_p$.

In this section, the ground ring will usually be $\Z$ with maximal ideal $\fm = (p) \subset \Z$. In that case, $L_0$ can be described as ``Ext-$p$-completion'' of abelian groups, given by
\[
L_0 M = \Ext^1_{\Z}(\Z/p^{\infty},M)
\]
as explained in \myref{localhom} or \cite[VI, \S 2]{holim}, or as ``analytic $p$-completion'', given by
\[
L_0 M = M \llb x \rrb / (x-p) M \llb x \rrb
\]
as described in \myref{rezkmodel} or \cite{analyticcompletion}. Analogous descriptions also hold for the ground ring $\Z_p$ with its maximal ideal $\fm = (p) \subset \Z_p$. For more details on $L$-completion, see Appendix \ref{appendix1}.

\subsection{Computations in $\la$-rings}

Since the notion of $\te$-ring depends on the choice of a prime $p$, it will be more convenient to work with $\la$-rings, whose definition does not involve a chosen prime. Recall that a $\la$-ring has underlying $\te^p$-rings for all primes $p$ (\myref{laforget}).

\begin{prop}\mylabel{FreeLambda}
The free lambda-ring on generators $\{ e_i \}_{i \in I}$ indexed by a set $I$ is the polynomial ring
\[
F^{\la}_{\Z} \left( \Z \lan e_i \mid i \in I \ran \right) = \Z [\la^k(e_i) \mid i \in I, k \geq 1].
\]
In particular, its underlying abelian group $T^{\la}_{\Z} \left( \Z \lan e_i \mid i \in I \ran \right)$ is a free abelian group. $T^{\la}_{\Z,n}$ picks up the weight $n$ piece, where elements are weighted by $\abs{\la^k(x)} = k$.
\end{prop}

\begin{proof}
The case of one generator is explained in \cite[\S 1.3]{lambda}. The general case follows from the fact that the free $\la$-ring functor $\Set \to \la \Alg_{\Z}$ preserves coproducts, and coproducts of $\la$-rings are computed as in commutative rings.
\end{proof}

\begin{example} 
Applying $T^{\la}_{\Z,n}$ to a free abelian group on one generator $e$ yields the free abelian group with basis described as follows:
\[
T^{\la}_{\Z,n} (\Z e) = \Z \lan \la^{k_1}(e) \la^{k_2}(e) \ldots \la^{k_j}(e) \mid k_1 \geq k_2 \geq \ldots \geq k_j \text{ and } k_1 + \ldots  + k_j = n \ran.
\]
For example, we have
\begin{align*}
T^{\la}_{\Z,1}(\Z e) &= \Z \lan \la^1(e) \ran \\
T^{\la}_{\Z,2}(\Z e) &= \Z \lan \la^2(e), \la^1(e) \la^1(e) \ran \\
T^{\la}_{\Z,3}(\Z e) &= \Z \lan \la^3(e), \la^2(e) \la^1(e), \la^1(e) \la^1(e) \la^1(e) \ran
\end{align*}
and more generally, the rank of $T^{\la}_{\Z,n}(\Z)$ is the number of partitions of $n$.
\end{example}

\begin{rmk}
The formulas
\[
T^{\la}_{R,n} (M \op N) \cong \bigoplus_{i+j = n} T^{\la}_{R,i} M \ot_{R} T^{\la}_{R,j} N
\]
come from the fact that $F^{\la}_R \colon \Mod_R \to \la\Alg_R$ preserves colimits, thus (finite) coproducts, which in $\la\Alg_R$ are the usual tensor products. In fact, $T^{\la}_R$ is exponential and in particular satisfies:
\[
T^{\la}_R(M \op N) \cong T^{\la}_R M \ot_{R} T^{\la}_R N. 
\]
\end{rmk}

\begin{prop}\mylabel{TnGenRel} Let $M$ be an abelian group. Then $T^{\la}_{\Z,2} M$ can be expressed in terms of generators and relations as
\[
T^{\la}_{\Z,n} M = \Z \lan \la^{k_1}(x_1) \la^{k_2}(x_2) \ldots \la^{k_j}(x_j) \mid x_i \in M \text{ and } k_1 + \ldots + k_j = n \ran / \sim
\]
where the relations $\sim$ are generated by the following:
\begin{itemize}
\item Terms $\la^{k_1}(x_1) \la^{k_2}(x_2) \ldots \la^{k_j}(x_j)$ that differ by a permutation of the factors are the same, e.g., $\la^1(x) \la^1(y) = \la^1(y) \la^1(x)$.
\item Using $\la^k(x+y) = \sum_{k'+k'' = k} \la^{k'}(x) \la^{k''}(y)$ and formally multiplying by factors so that the total degree is $n$, e.g., $\la^1(x+x') \la^1(y) = \la^1(x) \la^1(y) + \la^1(x') \la^1(y)$.
\end{itemize}
\end{prop}

\begin{example}\mylabel{T2GenRel}
\begin{align*}
T^{\la}_{\Z,2} M = \Z &\lan \la^2(x), \la^1(y) \la^1(z) \mid x, y, z \in M \ran / \sim \\
&\la^2(x+y) = \la^2(x) + \la^1(x) \la^1(y) + \la^2(y) \\
&\la^1(x) \la^1(y) \text{ is bilinear and symmetric.}
\end{align*}
\end{example}

\begin{lemma}\mylabel{lambda2NonLin}
In the notation of \myref{TnGenRel}, the following equalities hold.
\begin{enumerate}
\item $\la^k(0) = 0$ for all $k \geq 1$.
\item For any $n \in \Z$ and $x \in M$, we have
\[
\la^2(nx) = n \la^2(x) + \binom{n}{2} \la^1(x) \la^1(x)
\]
with the convention $\binom{n}{2} = \frac{n(n-1)}{2}$ in the case $n < 0$.
\end{enumerate}
\end{lemma}

\begin{proof}
(1) The equality $\la^1(0) = 0$ holds since $\la^1(x)$ is linear in $x$. The general case follows from
\begin{align*}
\la^k(0+0) &= \sum_{k'+k'' = k} \la^{k'}(0) \la^{k''}(0) \\
&= \la^k(0) + \la^k(0).
\end{align*}

(2) The formula clearly holds for $n=1$. The equality
\[
\la^2((n+1)x) = \la^2(nx) + \la^1(nx) \la^1(x) + \la^2(x)
\]
and induction on $n$ yields the result for $n \geq 1$. The equality $\la^2(0) = 0$ is the case $n=0$ of the formula. For negative integers, the equality 
\begin{align*}
0 &= \la^2(0) \\
&= \la^2(nx - nx) \\
&= \la^2(nx) + \la^1(nx) \la^1(-nx) + \la^2(-nx)
\end{align*}
gives the result.
\end{proof}

\begin{prop}\mylabel{T2Cyclic}
Let $m$ be a positive integer. Then we have
\[
T^{\la}_{\Z,2} (\Z/m) \simeq \begin{cases}
\Z/m \op \Z/m &\text{if } m \text{ is odd} \\
\Z/2m \op \Z/(m/2) &\text{if } m \text{ is even.} \\
\end{cases}
\]
When $m$ is odd, the generators can be taken to be $\la^1(1) \la^1(1)$ and $\la^2(1)$. When $m$ is even, the generators can be taken to be $\la^2(1)$ and $\la^1(1) \la^1(1) + 2 \la^2(1)$, respectively.
\end{prop}

\begin{proof}
Consider the exact sequence
\[
\xymatrix{
\Z \ar[r]^m & \Z \ar@{->>}[r] & \Z/m \ar[r] & 0 \\
}
\]
and turn it into the \emph{reflexive} coequalizer diagram
\[
\xymatrix{
\Z \op \Z \ar@<0.8ex>[r]^-{d_0 = (1,0)} \ar@<-0.8ex>[r]_-{d_1 = (1,m)} & \Z \ar[l] \ar@{->>}[r] & \Z/m. \\
}
\]
Applying $T_2 := T^{\la}_{\Z,2}$ yields the (reflexive) coequalizer diagram
\[
\xymatrix{
T_2(\Z \op \Z) \ar@<0.8ex>[r]^-{T_2 d_0} \ar@<-0.8ex>[r]_-{T_2 d_1} & \Z \ar@{->>}[r] & T_2 (\Z/m). \\
}
\]
Let $a = (1,0)$ and $b = (0,1)$ be the generators of $\Z \op \Z$. Then $T_2(\Z \op \Z)$ is the free abelian group
\[
T_2(\Z \op \Z) \cong \Z \lan \la^1(a) \la^1(a), \la^1(a) \la^1(b), \la^1(b) \la^1(b), \la^2(a), \la^2(b) \ran
\]
whereas $T_2 \Z \cong \Z \lan \la^1(1) \la^1(1), \la^2(1) \ran$ is free on two generators. The equality $d_0(a) = d_1(a) = 1$ implies that the maps $T_2 d_0$ and $T_2 d_1$ agree on generators involving only $a$ but not $b$, yielding no relations in $T_2(\Z/m)$. The remaining generators are sent to the following values:
\begin{align*}
&(T_2 d_0) \la^1(a) \la^1(b) = \la^1(d_0 a) \la^1(d_0 b) = \la^1(1) \la^1(0) = 0 \\
&(T_2 d_1) \la^1(a) \la^1(b) = \la^1(d_1 a) \la^1(d_1 b) = \la^1(1) \la^1(m) = m \la^1(1) \la^1(1) \\
&(T_2 d_0) \la^1(b) \la^1(b) = \la^1(d_0 b) \la^1(d_0 b) = \la^1(0) \la^1(0) = 0 \\
&(T_2 d_1) \la^1(b) \la^1(b) = \la^1(d_1 b) \la^1(d_1 b) = \la^1(m) \la^1(m) = m^2 \la^1(1) \la^1(1) \\
&(T_2 d_0) \la^2(b) = \la^2(d_0 b) = \la^2(0) = 0 \\
&(T_2 d_1) \la^2(b) = \la^2(d_1 b) = \la^2(m)
\end{align*}
which implies
\[
T_2 (\Z/m) = \Z \lan \la^1(1)\la^1(1), \la^2(1) \ran / \left\{ m \la^1(1) \la^1(1) = 0, \la^2(m) = 0 \right\}.
\]
Using $\la^2(m) = m \la^2(1) + \binom{m}{2} \la^1(1) \la^1(1)$, we can rewrite this as
\[
T_2 (\Z/m) = \Z \lan \la^1(1) \la^1(1), \la^2(1) \ran / \left\{ m \la^1(1) \la^1(1) = 0, m \la^2(1) = - \binom{m}{2} \la^1(1) \la^1(1) \right\}.
\]
We now distinguish two cases.

\textbf{Case $m$ odd.} Then $\binom{m}{2}= \frac{m(m-1)}{2}$ is a multiple of $m$, so that the relation on $\la^2(1)$ is $m \la^2(1) = 0$ and we conclude
\[
T_2 (\Z/m) \cong \Z/m \op \Z/m
\]
with generators $\la^1(1) \la^1(1)$ and $\la^2(1)$.

\textbf{Case $m$ even.} Then $\binom{m}{2}$ is a multiple of $m-1$ and the relation on $\la^2(1)$ is
\begin{align*}
m \la^2(1) &= - \binom{m}{2} \la^1(1) \la^1(1) \\
&= \frac{m}{2} (1-m) \la^1(1) \la^1(1) \\
&= \frac{m}{2} \la^1(1) \la^1(1)
\end{align*}
which implies that $\la^2(1)$ has order $2m$. Note that $\la^1(1) \la^1(1) + 2 \la^2(1)$ has order $\frac{m}{2}$. One readily checks that the map
\begin{align*}
\Z/2m \op \Z/(m/2) &\surj T_2 (\Z/m) \\
(1,0) &\mapsto \la^2(1) \\
(0,1) &\mapsto \la^1(1) \la^1(1) + 2 \la^2(1)
\end{align*}
is an isomorphism.
\end{proof}

\begin{thm}\mylabel{TLambda}
The functor $T^{\la}_{\Z} \colon \Mod_{\Z} \to \Mod_{\Z}$ preserves $L_0$-equivalences.
\end{thm}

\begin{proof}
By \myref{CompletedMonad} \eqref{ColimPreserv}, it suffices to show that each functor $T^{\la}_{\Z,n} \colon \Mod_{\Z} \to \Mod_{\Z}$ preserves $L_0$-equivalences. 

Note that $T^{\la}_{\Z}$ satisfies the same formal properties as $\T$, namely it is a graded exponential monad which preserves filtered colimits and reflexive coequalizers, and sends finite free modules to finite free modules. Hence, the reduction step in the proof of \myref{MainThmRestated} holds, and it suffices to show the analogue of \myref{key} for $T^{\la}_{\Z}$. Explicitly: there exists a positive integer $k=k(n)$ such that the natural map
\begin{equation}\label{WantIso}
\Z/p \otimes T^{\la}_{\Z,n}(\Z) \xrightarrow{} \Z/p \otimes T^{\la}_{\Z,n}(\Z/p^k)
\end{equation}
is an isomorphism.

The cases $n=0$ and $n=1$ are trivial; one may take $k(0) = 1$ and $k(1) = 1$. Moreover, surjectivity is automatic, since both $T^{\la}_{\Z,n}$ and $\Z/p \otimes -$ preserve surjectivity, by \myref{ReflexCoeq}.

To illustrate the main ideas, we focus on the special case $n=2$; the general result is proved similarly.

\textbf{Case n=2.} Let us show that $k(2) = 2$ works. Writing $T_2 := T^{\la}_{\Z,2}$, the domain of \eqref{WantIso} is
\[
\Z/p \ot T_2 \Z = \Z/p \lan \la^2(1) \ran \op \Z/p \lan \la^1(1) \la^1(1) \ran
\]
whereas the target is computed using \myref{T2Cyclic}. If $p$ is odd, we have:
\begin{align*}
\Z/p \ot T_2 (\Z/p^2) &= \Z/p \ot \left( \Z/p^2 \lan \la^2(1) \ran \op \Z/p^2 \lan \la^1(1) \la^1(1) \ran  \right) \\
&= \Z/p \lan \la^2(1) \ran \op \Z/p \lan \la^1(1) \la^1(1) \ran
\end{align*}
and in the case $p=2$, we have:
\begin{align*}
\Z/2 \ot T_2 (\Z/4) &= \Z/2 \ot \left( \Z/8 \lan \la^2(1) \ran \op \Z/2 \lan \la^1(1) \la^1(1) + 2 \la^2(1) \ran \right) \\
&= \Z/2 \lan \la^2(1) \ran \op \Z/2 \lan \la^1(1) \la^1(1) \ran.
\end{align*}
The map \eqref{WantIso} sends the generators where their names suggest, and is in particular an isomorphism of $2$-dimensional $\Z/p$-vector spaces.
\end{proof}

\begin{rmk}\mylabel{nonlinear}
The previous proof shows that $k(2)$ must be greater than $1$, hence that $\T_n$  does not commute with $E_*/\fm \otimes -$ in general.  Furthermore, the results of Appendix \ref{appendix2} imply that $k(n)$ can be taken to be $p(n)$, the number of partitions of a set with $n$ elements, independent of the chosen prime. More generally, at height $h$, $k(n)$ is bounded above by the rank of $E^0(B\Sigma_n)$ as computed by Strickland \cite{symmgps}.
\end{rmk}

Taking $R = \Z$, \myref{TLambda} and \myref{CompletionLambda} yield the following.

\begin{cor}\mylabel{TLambdaZp}
The functor $T^{\la}_{\Z_p} \colon \Mod_{\Z_p} \to \Mod_{\Z_p}$ preserves $L_0$-equivalences.
\end{cor}

\subsection{From $\lambda$-rings to $\theta$-rings}

\begin{prop}\mylabel{pLocalRetract}
Let $R$ be a $\la$-ring. Consider the monads $T^{\te}_R$ and $T^{\la}_R$ on $\Mod_R$, and the natural transformation $\phy \colon T^{\te}_R \to T^{\la}_R$. If $R$ is $p$-local, then $\phy$ is a split monomorphism, i.e., admits a retraction. 
\end{prop}

\begin{proof}
Let $U^{\la}_{\Alg} \colon \la \Alg_R \to \Alg_R$ denote the forgetful functor, $F^{\la}_{\Alg}$ its left adjoint, and $G^{\la}_{\Alg}$ its right adjoint.To show the claim, it suffices to show that the canonical map
\[
T^{\te}_{\Alg} := U^{\te}_{\Alg} F^{\te}_{\Alg} \to U^{\la}_{\Alg} F^{\la}_{\Alg} =: T^{\la}_{\Alg}
\]
of monads on $\Alg_R$ is a split monomorphism. By adjunction, such natural transformations $T^{\te}_{\text{Alg}} \to T^{\la}_{\text{Alg}}$ correspond to maps
\[
U^{\la}_{\Alg} G^{\la}_{\Alg} \to U^{\te}_{\Alg} G^{\te}_{\Alg}
\]
i.e., natural transformations $C^{\la} \to C^{\te}$ of the comonads, in the reverse direction. Under that correspondence, a map $T^{\te}_{\Alg} \to T^{\la}_{\Alg}$ is a split monomorphism if and only if the corresponding map $C^{\la} \to C^{\te}$ is a split epimorphism, i.e., admits a section.

The comonad $C^{\la} = U^{\la}_{\Alg} G^{\la}_{\Alg} \colon \Alg_R \to \Alg_R$ associated to $\la$-rings is naturally isomorphic to the Witt vectors comonad $\W$ \cite[4.16]{lambda}, while the comonad $C^{\te} = U^{\te}_{\Alg} G^{\te}_{\Alg}$ associated to $\te$-rings is naturally isomorphic to the $p$-typical Witt vectors comonad $\W_p$ \cite[2.13, 3.3]{pleth}.

Since $R$ is $p$-local, any $R$-algebra $A$ is $p$-local and in that case the canonical map $\W(A) \to \W_p(A)$ admits a section 
\[
\W_p(A) \ral{s^n} \W(A)
\]
for each natural number $n$ coprime to $p$, i.e., $(n,p) = 1$, as explained in \cite[4.4.9, 4.4.10]{formal} and also in \cite[Chapter III]{demazure}. Taken together, these sections exhibit a natural isomorphism of rings
\[
\W(A) \cong \prod_{(n,p)=1} \W_p(A),
\]
and hence a splitting of $C^{\la} \to C^{\te}$. 
\end{proof}

\begin{thm}\mylabel{TThetaZp}
The free $\Z_p$-$\te$-algebra monad $T^{\te}_{\Z_p} \colon \Mod_{\Z_p} \to \Mod_{\Z_p}$ preserves $L_0$-equivalences.
\end{thm}

\begin{proof}
By \myref{TLambdaZp}, the monad $T^{\la}_{\Z_p} \colon \Mod_{\Z_p} \to \Mod_{\Z_p}$ preserves $L_0$-equiva\-lences. By \myref{pLocalRetract}, $T^{\te}_{\Z_p}$ is a retract of $T^{\la}_{\Z_p}$ since $\Z_p$ is $p$-local. By \myref{CompletedMonad} \eqref{RetractPreserv}, $T^{\te}_{\Z_p}$ preserves $L_0$-equivalences.
\end{proof}

\appendix

\section{$L$-completion}\label{appendix1}

For the convenience of the reader, we survey the essential structural features of the category of complete modules we are working with. Our exposition is inspired by and extends the treatments in \cite[Appendix A]{moravak}, \cite{filteredcolim}, and \cite{hoveynotes}, as well as the original source \cite{gmcomplete}. More recently, Rezk \cite{analyticcompletion} has offered a different perspective on completion, extending the theory to simplicial modules and algebras.

\subsection{$\fm$-adic completion and its derived functors}

Most statements in this section hold true for a noetherian commutative ring $R$ and an ideal $I \subset R$ which is generated by a regular sequence or, more generally, non-noetherian rings with some weaker finiteness conditions with respect to $I$, cf. \cite{gmcomplete}. For simplicity and with an eye towards our applications, however, we will restrict ourselves to the case of a regular local noetherian ring $R$ and take $I=\fm$, the maximal ideal of $R$. The number of generators of $I$ will be denoted by $h$. In the special case $R = E_*$, this is consistent with our earlier convention. 

For $M$ any $R$-module, completion with respect to $\fm$ is defined to be $M_{\fm}^{\wedge} = \varprojlim M/\fm^k$, and $M$ is called \Def{$\fm$-complete} if the canonical map $M \to M_{\fm}^{\wedge}$ is an isomorphism. By the Artin-Rees Lemma, $\fm$-adic completion $(-)_{\fm}^{\wedge}$ is an exact functor on the full subcategory of $\Mod_R$ of finitely generated modules, and the $\fm$-adic completion of a \emph{finitely generated} module $M$ can be constructed as $R_{\fm}^{\wedge} \otimes M$. In particular, $R_{\fm}^{\wedge}$ is flat over $R$. However, viewed as a functor on all of $\Mod_R$, $(-)_{\fm}^{\wedge}$ is badly behaved homologically: In general, it is neither left nor right exact. This can be used to show that the full subcategory of $\fm$-complete modules is not abelian; in fact, because the coimage is not necessarily complete, there exists an example \cite[74.6.1.]{stacksproject} of a map $f$ between complete modules with $\mathrm{im}(f) \ncong \mathrm{coim}(f)$.

Since the modules coming from topology are often not finitely generated, Greenlees and May \cite{gmcomplete} propose to work with the zeroth left derived functor of $\fm$-adic completion instead. 

\begin{defn}
Let $L_s$ be the $s$-th left\footnote{The right derived functors of $\fm$-adic completion can be shown to vanish, see \cite{gmcomplete}.} derived functor of the $\fm$-adic completion functor $(-)_{\fm}^{\wedge}$, where $s \ge 0$. For any $M \in \Mod_R$, the natural map $M \to M_{\fm}^{\wedge}$ factors canonically as $M \xrightarrow{\eta_M} L_0M \to M_{\fm}^{\wedge}$, and $M$ is called \Def{$L$-complete} if $\eta_M$ is an isomorphism. 
\end{defn}

In \cite{gmcomplete}, Greenlees and May initiated the study of these functors. In particular, they establish a duality between the $L_s$ and Grothendieck's local cohomology, and prove their basic properties. To start with, derived completion can be related to ordinary completion via higher $\mathrm{Tor}$-functors, as follows. 

\begin{prop}\mylabel{ses}
For any $s \ge 0$ and $M \in \Mod_{R}$ there is a short exact sequence
\[ 0\to {\varprojlim}_k^1 \mathrm{Tor}_{s+1}^R(R/\fm^k,M) \to L_sM \to  {\varprojlim}_k \mathrm{Tor}_s^R(R/\fm^k,M) \to 0.\]
\end{prop}

\begin{rmk}
In fact \cite[1.1]{gmcomplete}, this result is true for any commutative ring $R$ and the left derived functors of $I$-adic completion for any ideal $I \subset R$; it does not rely on the identification of $L_s$ with local homology. 
\end{rmk}

\begin{thm}[Greenlees-May duality]\mylabel{localhom}
The $s$-th left-derived functor of $\fm$-adic completion can be computed by local cohomology, i.e., for $M \in \Mod_{R}$,
\[ L_sM = \mathrm{Ext}_R^{h-s}(H_{\fm}^h(R), M) \]
where $H_{\fm}^i(R) \cong R/\fm^{\infty}$ non-canonically if $i=h$, and $0$ otherwise. In particular, $L_s=0$ if $s >h$.
\end{thm}

\begin{rmk}\mylabel{rezkmodel}
In \cite{analyticcompletion}, Rezk gives a different model for $L$-completion with respect to a maximal ideal $\fm$ generated by elements $u_1,\ldots,u_h$. He shows that the $L$-completion of a module $M \in \Mod_R$ can be identified with \Def{analytic completion}, 
\[ L_0M = M \llbracket x_1,\ldots,x_h\rrbracket/(x_1-u_1,\ldots,x_h-u_h).\] 
The properties of $L$-complete modules can then be developed in this context as well, and this other point of view turns out to be useful in many situations. 
\end{rmk}

We can now list the salient features of the derived completion functors $L_s$ that are relevant for our applications, see \cite[A.4, A.6]{moravak}.

\begin{prop}\mylabel{lprop}
Let $M \in \Mod_{R}$. 
\begin{enumerate}
 \item If $M$ is of the form $N_{\fm}^{\wedge}$ or $L_sN$ for $s\ge 0$, then the natural map $M \to L_0M$ is an isomorphism and $L_tM=0$ for all $t > 0$. In particular, $L_0$ is idempotent.
 \item The natural map 
 \[ R/\fm^k \otimes M \to R/\fm^k \otimes L_0M\]
 is an isomorphism. 
 \item $L_0M = 0$ if and only if $R/\fm \otimes M = 0$. 
\end{enumerate}
\end{prop}

\begin{rmk}
Both $I$-adic and derived completion are idempotent for noetherian rings \cite[3.6]{noetherian} \cite{analyticcompletion}. However, $I$-adic completion need not be idempotent in general \cite{noetherian}. We do not know whether the zeroth derived functor $L_0$ is always idempotent.
\end{rmk}

We end this section with a useful criterion for when a morphism between $L$-complete modules is an isomorphism. Since we do not know of a published reference for this probably well-known fact, we include a proof.

\begin{lemma}\mylabel{detection}
Suppose $f\colon M \to N$ is a map of $R$-modules with $N$ flat. If $R/\fm \otimes f$ is an isomorphism, then so is $L_0f\colon  L_0M \to L_0N$. 
\end{lemma}
\begin{proof}
We first prove surjectivity. To this end, consider the exact sequence $M \xrightarrow{f} N \to C \to 0$. Since $L_0$ is right exact, $L_0M \xrightarrow{L_0f} L_0N \to L_0C \to 0$ is also exact. Reducing mod $\fm$ yields the exact sequence 
\[ R/\fm \otimes M \xrightarrow{R/\fm \otimes f} R/\fm \otimes N \to R/\fm \otimes C \to 0. \] 
So, by assumption, $R/\fm \otimes C =0$, which implies $L_0C=0$ by \myref{lprop}. The claim follows.  

Injectivity is proved similarly: Let $K$ be the kernel of $L_0f$; using the first part of the lemma, we get an exact sequence
\[
0 \to K \to L_0M \to L_0N \to 0.
\]
Since $N$ is flat, so is $L_0N$ by \myref{flatchar} and hence $\mathrm{Tor}_1^R(L_0N, R/\fm)=0$. Therefore, tensoring with $R/\fm$ yields the exact sequence
\[
0 \to R/\fm \otimes K \to R/\fm \otimes M \to R/\fm \otimes N \to 0
\]
and we get $R/\fm \otimes K = 0$ by assumption. Since $K$ is $L$-complete, \myref{lprop} implies $K = 0$.
\end{proof}

\subsection{The category of $L$-complete modules} The category of $L$-complete $E_*$-modules plays a pivotal role in the algebraic theory of power operations for Morava $E$-theory, since it is the natural target of completed $E$-homology; see \myref{homotknlocal}. As it turns out this category has many good and some exotic properties, which we review in this section. 

For simplicity, let $R$ be a regular noetherian ring with a fixed maximal ideal $\fm$, usually assumed to be $\fm$-complete. 

\begin{note}
Let $\Modhat_{R}$ denote the full subcategory of $\Mod_{R}$ consisting of $L$-complete modules, with the natural embedding $\iota\colon  \Modhat_{R} \hookrightarrow \Mod_{R}$.
\end{note}

The embedding $\iota$ has left adjoint $L_0 \colon \Mod_{R} \to \Modhat_{R}$, which is also its left-inverse, i.e., $L_0 \circ \iota = \id$. It follows from this and the properties (\myref{lprop}) of the completion functor $L_0$ that $\Modhat_{R}$ is bicomplete: while limits are computed in $\Mod_{R}$ via $\iota$, colimits in $\Modhat_{R}$, denoted by $\lcolim$, are constructed as $L_0\colim$, the latter colimit being taken in $\Mod_{R}$.

Unlike the category of $\fm$-complete modules, $\Modhat_{R}$ is an abelian subcategory of $\Mod_{R}$ closed under extensions and $\mathrm{Ext}_{R}^{*}$, see \cite[A.6]{moravak}. In particular, $\mathrm{Hom}_{R}(M,N)$ is $L$-complete for $L$-complete modules $M$ and $N$, as are all extensions between $M$ and $N$.   Using this and defining the $L$-completed tensor product $ - \hat{\otimes} - = L_0(- \otimes -)$, it is easy to verify that $\Modhat_{R}$ is closed symmetric monoidal under $\hat{\otimes}$. Furthermore, \myref{flatchar} implies that $\Modhat_{R}$ has enough projectives, although no non-zero injective objects. To see the latter claim, let $I \in \Modhat_R$ be injective. Applying $\Hom(-,I)$ to the monomorphism
\[ 0 \to R \xrightarrow{x} R\]
for any $x \in \fm$ shows that $\fm I = I$ and hence $I=0$ by \myref{lprop}(3). In summary:

\begin{thm}\mylabel{lcompletecat}
The category of $L$-complete modules $\Modhat_{R}$ is a bicomplete closed symmetric monoidal category with enough projectives. The adjunction
\[ L_0\colon  \Mod_{R} \rightleftarrows \Modhat_{R} \colon \iota \]
exhibits the category of $L$-complete modules as a full abelian subcategory of all $R$-modules, closed under limits and extensions. 
\end{thm}

\begin{rmk}\mylabel{abelianapprox}
In fact, in the language of \cite{approxsubcats}, $\Modhat_{R}$ is the \Def{best reflective abelian approximation} to the category of $\fm$-complete $R$-modules.
\end{rmk}

As a consequence, we note that $\iota$ preserves reflexive coequalizers; this will become important in the proof of our main theorem. To show this, recall that the coequalizer of a pair $(f,g)\colon  A \rightrightarrows B$ in an abelian category can be computed as the cokernel of $f-g$. Conversely, every cokernel diagram $A \xrightarrow{f} B \to C \to 0$ can be written as a coequalizer diagram $(f+id_B, id_B)\colon  A \oplus B \rightrightarrows B \to C \to 0$, which is reflexive\footnote{In fact, every coequalizer diagram in a category with direct sums is equivalent to a reflexive one, by simply adding an extra copy of the codomain to the domain and using the obvious morphisms; see \myref{ReflexCoeq}.} with the natural inclusion $s\colon  B \to  A \oplus B$ as section. Since $\Modhat_{R}$ is an abelian subcategory of $\Mod_{R}$, $\iota$ preserves cokernels and thus reflexive coequalizers.  

As mentioned above, colimits of complete modules in $\Mod_{R}$ are not necessarily complete and thus need to be completed. But $L_0$ is not exact, whence neither are direct sums\footnote{For a natural example, see \cite[1.3]{filteredcolim}.} and filtered colimits in $\Modhat_{R}$, so $\Modhat_{R}$ is not a Grothendieck abelian category. However, Hovey \cite[1.1]{filteredcolim} shows that for $\cD$ a discrete or filtered diagram, the $s$-th left derived functor of $\lcolim_{\cD}$ is given by $L_s \colim_{\cD}$. Moreover, if $\cD$ is discrete, then $L_h \colim_{\cD} = 0$.

\begin{rmk}\mylabel{small} 
Another noticeable feature of $\Modhat_{R}$ is its lack of small objects: Since $R$ itself is $L$-complete, so are all finite free modules, but these are not small in $\Modhat_{R}$. Indeed, if $R^{\oplus r}$ was small  for some finite $r>0$, the Freyd-Mitchell theorem would produce an equivalence $\Modhat_{R} \simeq \Mod_{\mathrm{End}(R^{\oplus r})^{\mathrm{op}}} \simeq \Mod_{R}$, since $R^{\oplus r} \in \Modhat_{R}$ is a projective generator. 
\end{rmk}

This remark should be compared with the following result that informally speaking states that $R \in \Modhat_R$, while not being small, is not too large either. 

\begin{prop}\mylabel{ksmall}
The category $\Modhat_R$ is $\kappa$-presentable for $\kappa$ any regular cardinal larger than $\aleph_0$.
\end{prop}
\begin{proof}
Let $\kappa$ be a regular cardinal larger than $\aleph_0$, then it suffices to show that $R$ is $\kappa$-small. To this end, we use Rezk's construction of the $L$-completion of $M \in \Mod_R$,
\[ L_0M = M \llbracket x_1,\ldots,x_h \rrbracket / (x_1-u_1,\ldots,x_h-u_h), \]
as in \myref{rezkmodel}. Since $\kappa$-filtered colimits commute with limits of size $\kappa' < \kappa$ in $\mathrm{Set}$, the same holds in $\Mod_R$. In particular, countable products commute with $\kappa$-filtered colimits, hence $L_0$ preserves $\kappa$-filtered colimits. Therefore, for a $\kappa$-filtered colimit diagram $M \colon I \to \Modhat_R$, we get
\begin{eqnarray*}
\Hom_{\Modhat_R}(R, \lcolim\ M_i) & = & \Hom_{R}(R, L_0\colim\ M_i) \\
& = & \Hom_R(R, \colim\ M_i) \\
& = & \colim\ \Hom_R(R, M_i) \\
& = & \colim\ \Hom_{\Modhat_R}(R,M_i).
\end{eqnarray*}
This gives the claim. 
\end{proof}

\begin{rmk}
Using \myref{ksmall} and \cite[9.3(1)]{analyticcompletion}, one can show that Rezk's model structure on the category $\Chhat$ of (unbounded) chain complexes of $L$-complete $R$-modules coincides with the one created from the Quillen model structure on $\Ch$ through $\iota$. Consequently, this model structure is cofibrantly generated and monoidal. The associated derived category of $L$-complete modules has been studied independently in recent work of Valenzuela \cite{val}.
\end{rmk}

\subsection{Flat modules}\label{ss:flat}

Let $\Modflat_{R}$ be the full subcategory of $\Mod_{R}$ on the flat $R$-modules. By Lazard's \myref{lazardalg}, $M \in \Mod_{R}$ is flat if and only if the comma category $(\Modff_{R})_{/M}$ is filtered, which implies that $M$ can be written canonically as a filtered colimit of finite free modules. We will use this fact frequently throughout this document. 

While flat modules are not necessarily complete, as for example $R^{\oplus \omega}$, there is the following characterization of flat and $L$-complete $R$-modules. 

\begin{prop}\mylabel{flatchar}
The following properties of an $L$-complete $R$-module $M$ are equivalent:
\begin{enumerate}
 \item $M$ is projective in $\Modhat_{R}$.
 \item $M$ is pro-free, i.e., it is of the form $L_0F$ for some free $F \in \Mod_{R}$.
 \item $M$ is flat as an object in $\Mod_{R}$.
\end{enumerate}
Furthermore, if $N \in \Modflat_{R}$, then $L_0N$ is pro-free and hence flat as an $R$-module. 
\end{prop}
\begin{proof}
The equivalence of (1) and (2) is \cite[A.9]{moravak}. Now assume (2) and let $M = R^{\oplus I}$ be a free module. By \cite[A.13]{moravak}, the canonical map $L_0(R^{\oplus I}) \to \prod_I R$ exhibits $L_0M$ as a retract of $ \prod_I R$. Because $R$ is noetherian, products of flat modules are flat, which in particular applies to $\prod_I R$, so (3) follows. Conversely, assume $M$ is flat, then $\mathrm{Tor}^{R}_s(R/\fm, M)=0$ for all $s>0$. Applying \cite[A.9]{moravak} once more gives (2).

The last assertion is due to Hovey and can be found in his unpublished notes \cite[1.2]{hoveynotes}; for completeness, we sketch the argument. By \cite[A.9]{moravak}, it suffices to show that $\mathrm{Tor}^R_1(R/\fm,L_0N) =0$ for all $s>0$. To this end, let $F_{\bullet} \to R/\fm \to 0$ be a resolution of $R/\fm$ by finite free modules. Using \cite[A.4]{moravak}, the claim then reduces to $H_sL_0(F_{\bullet} \otimes N) = 0$.
Denoting the image of $F_i$ in $F_{i-1}$ by $K_i$, flatness of $N$ yields short exact sequences
\[ 0 \to K_i \otimes N \to F_i \otimes N \to K_{i-1} \otimes N \to 0. \]
Therefore, again using flatness of $N$, we get for $s >0$ and all $i$
\[ L_{s+1}(K_{i-1} \otimes N) \cong L_s(K_i \otimes N)\]
as well as an exact sequence
\[ 0 \to L_1(K_{i-1}\otimes N) \to L_0(K_i \otimes  N) \to L_0(F_{i-1} \otimes N) \to L_0(K_{i-1} \otimes N) \to 0.\]
Note that $K_0 \ot N = R/\fm \ot N$ is $\fm$-complete, hence $L$-complete, and thus satisfies $L_s(K_0 \otimes N) = 0$ for $s>0$. It follows inductively that $L_s(K_i \otimes N) = 0$ holds for all $s>0$ and all $i$, yielding the short exact sequence
\[ 0 \to L_0(K_i \otimes N) \to L_0(F_{i-1} \otimes N) \to L_0(K_{i-1} \otimes N) \to 0,\]
hence $H_sL_0(F_{\bullet} \otimes N) = 0$ for $s>0$.
\end{proof}

 An $L$-complete module $M$ is said to be \Def{$L$-flat} if it is a flat object in $\Modhat_R$. In contrast to the previous result, there are very few $L$-flat modules.

\begin{prop}
If $R$ has dimension 1, then $M \in \Modhat_R$ is $L$-flat if and only if it is pro-free. If $\dim R \ge 2$, then being $L$-flat is equivalent to being finite free. 
\end{prop}
\begin{proof}
Let $\Torhat_1(M,-)$ be the first left derived functor of $M \hat{\otimes} -$; on finitely generated modules, this functor agrees with the usual Tor functor. Since $M$ is $L$-flat and $R/\fm$ is finitely generated, we obtain
\[ 0 = \Torhat_1^R(M,R/\fm) = \mathrm{Tor}_1^R(M,R/\fm),\]
so $M$ is pro-free by \cite[A.9]{moravak}. If $\dim R =1$, the converse holds since coproducts are exact on $\Modhat_R$ in this case. 

If $\dim R \ge 2$, consider the monomorphism of $L$-complete modules 
\[f\colon K=\prod_{k\ge 1}(u_1^k,\ldots,u_h^k) \hookrightarrow (\bigoplus_{k \ge 1} R)^{\wedge}_{\fm} = N.\]
If $M$ is an $L$-flat module, it is pro-free by the above, so we can write $M=L_0(\oplus_I R)$ for some indexing set $I$. Applying $M \hat{\otimes} -$ to $f$ gives the morphism $L_0(\oplus_I K) \to L_0(\oplus_I N)$ which we will show to be injective if and only if $I$ is finite. Finiteness is clearly sufficient. For the converse it is enough to consider the case of countable infinite $I$, because retracts of flat objects are flat. It is then easy to check that the argument given in \cite[Appendix B]{completehopfalgebroids} generalizes to $\dim R \ge 2$, giving that $L_1(N/K) \ne 0$. Therefore, $M=L_0(\oplus_I R)$ cannot be $L$-flat for infinite $I$. 
\end{proof}

\begin{rmk}
Consequently, projective objects are generally not flat in $\Modhat_{R}$, which implies that the derived functors of $\hat{\otimes}$ are not balanced; this issue is discussed further in \cite[\S 1.5]{hoveynotes}. 
\nc
\end{rmk}

In fact, the last part of \myref{flatchar} shows that $\iota L_0$ restricts to an endofunctor of $\Modflat_{R}$. Informally speaking, passage to $L$-complete modules has the effect of identifying the free $R$-modules with the flat ones. 

\begin{cor}\mylabel{flatcomp}
When restricted to $\Modflat_{R}$, $L_0$ coincides with $\fm$-adic completion, and $L_s=0$ for $s>0$.
\end{cor}
\begin{proof}
This is an immediate consequence of \myref{ses}.
\end{proof}

The situation is summarized in the following commutative diagram, in which the superscript $\mathrm{proj}$ indicates the full subcategory of projective objects:

\[\xymatrix{
\Modff_{R} \ar@{^{(}->}[r] \ar[d]_{\sim} & \mathrm{Ind}\ \Modff_{R} \ar@{=}[r]  & \Modflat_{R} \ar@{^{(}->}[r] \ar@<-.75ex>[d]_{L_0=(-)_{\fm}^{\wedge}}  & \Mod_{R} \ar@<-.75ex>[d]_{L_0} \\
\Modhat_{R}^{\mathrm{ff}}  \ar@{^{(}->}[rr] & & \Modhat_{R}^{\mathrm{proj}} \ar@{^{(}->}[r] \ar@<-.75ex>[u]_{\iota} & \Modhat_{R}. \ar@<-.75ex>[u]_{\iota}}
\]

\section{Representation theory}\label{appendix2}

In this appendix, we prove a representation-theoretic integral analogue of \myref{nilp} at height $h=1$, one of the key ingredients to our main theorem. We include this curious result not only because it provided some motivation for the general proof, but also because it might shed some additional light on the inner workings of the algebraic approximation functors $\T_n$.

To this end, note that the multiplication by $p$ map can be lifted to homotopy theory as
\[ S^0\xrightarrow{\Delta} \bigvee_{i=1}^p S^0 \xrightarrow{\nabla} S^0 \]
where $\Delta$ is the diagonal and $\nabla$ is the fold map. Applying $K(E\Sigma_{m+} \wedge_{\Sigma_m} -)$ to this diagram yields the map
\[ t(m,p)\colon  \rep(\Sigma_m) \xrightarrow{\mathrm{Res}} \bigoplus_{i_1+ \cdots + i_p=m} \rep(\Sigma_{i_1} \times \cdots \times \Sigma_{i_m}) \xrightarrow{\mathrm{Ind}} \rep(\Sigma_m) \]
which is an integral version of $\T_m(p)\colon  K(B\Sigma_m) \to K(B\Sigma_m)$. Here, $\mathrm{Res}$ and $\mathrm{Ind}$ denote the direct sums of restriction and induction for the corresponding subgroups, respectively. 

\begin{rmk}
We will keep the notation from before, but emphasize that the parameter $p$ does not have to be a prime in this section. Therefore, \myref{mainreptheory} is a generalization of \myref{nilp}.
\end{rmk}

\subsection{A few examples}

Before we discuss the general result in the next section, we give three explicit examples. 

\begin{example}
For all integers $p$, the map $t(2,p)\colon  \rep(\Sigma_2) \to \rep(\Sigma_2)$ is nilpotent mod $p$.
\end{example}
\begin{proof}
Note that the cases $p>2$ or $p=1$ are straightforward, so let us consider $p=2$. 

Recall that $\rep(\Sigma_2) = \Z[\sigma]/(\sigma^2-1)$, where $\sigma$ the sign representation; also, if $\rho$ denotes the regular representation, $\rho = 1 + \sigma$. Now $\mathrm{Ind}$ on the three factors is either the identity or induction from the trivial group, i.e., $\rho \otimes -$. After these recollections, we are now ready to compute the composite $\mathrm{Ind} \circ \mathrm{Res}$ on the basis elements $1,\sigma$ of the representation ring:
\[ \mathrm{Ind} \circ \mathrm{Res}(1) = \mathrm{Ind}(1_{\Sigma_2} + 1_{\Sigma_1} + 1_{\Sigma_2}) = 1 + \rho + 1 = 1 + 1 + \sigma +1 = 3 +\sigma \]
and
\[ \mathrm{Ind} \circ \mathrm{Res}(\sigma) = \mathrm{Ind}(\sigma + 1_{\Sigma_1} + \sigma) = \sigma + \rho + \sigma = \sigma + 1 + \sigma + \sigma = 1 + 3\sigma. \]
Therefore, the matrix corresponding to $t(2,2)$ in terms of the basis $(1,\sigma)$ is 
\[ t(2,2) = \begin{pmatrix}
 3 & 1 \\
 1 & 3
\end{pmatrix}\] 
with characteristic polynomial $\chi_{t(2,2)}(x)=(x-2)(x-4) \equiv_2 x^2$. Thus $t(2,2)^2 \equiv 0 \mod 2$.
\end{proof}

\begin{example}
Likewise, using the basis of $\rep(\Sigma_3)$ given by the irreducible representations $1, \sigma, V$ with $V$ the standard representation, one computes that $t(3,3)\colon  \rep(\Sigma_3) \to \rep(\Sigma_3)$ is given by the matrix
\[ t(3,3) = \begin{pmatrix} 10 & 1 & 8 \\ 
1 & 10 & 8 \\
8 & 8 & 19\end{pmatrix}.\]
Therefore, $\chi_{t(3,3)}(x)=(x-27)(x-9)(x-3) \equiv_3 x^3$, so $t(3,3)$ is nilpotent mod $3$. The remaining cases of $t(3,p)$ for $p \ne 3$ are easy. 
\end{example}

Finally, to convince the reader that our proof below works for positive integers $p$ that are not prime, we give one more example. 

\begin{example}
The representation ring $\rep(\Sigma_4)$ has a basis of irreducible representations $(1, \sigma, T, V, W)$, where $V$ is the standard representation, $W = V \otimes \sigma$, and $T$ is the irreducible 2-dimensional representation coming from the standard representation of $\Sigma_3$ via the exact sequence 
\[0 \to \Z/2 \times \Z/2 \to \Sigma_4 \to \Sigma_3 \to 0.\]
With respect to this basis, 
\[t(4,4)= 
\begin{pmatrix}
35 & 1 & 20 & 45 & 15 \\
1 & 35 & 20 & 15 & 45 \\
20 & 20 & 56 & 60 & 60 \\
45 & 15 & 60 & 115 & 81 \\
15 & 45 & 60 & 81 & 115
\end{pmatrix}\]
which has characteristic polynomial 
\[\chi_{t(4,4)}(x)=(x-4^4)(x-4^3)(x-4^2)(x-4^2)(x-4).\]
Consequently, $t(4,4)^5 \equiv 0 \mod 4$.
\end{example}

\subsection{The general result}

In this section we give a proof of the following claim, which generalizes \myref{nilp} in the height 1 case.

\begin{prop}\mylabel{mainreptheory}
For any pair of positive integers $m$ and $p$, the map
\[ t = t(m,p) \colon  \rep(\Sigma_m) \xrightarrow{\mathrm{Res}} \bigoplus_{i_1+\cdots+i_p=m} \rep(\Sigma_{i_1} \times \cdots \times \Sigma_{i_p}) \xrightarrow{\mathrm{Ind}} \rep(\Sigma_m)\]
is nilpotent mod $p$. 
\end{prop}
\begin{proof}
Let $\Lambda$ be the free $\Z$-module on the basis given by characteristic functions on the conjugacy classes of $\Sigma_m$. Since all characters of symmetric groups are integer-valued, we get an injection 
\[ \rep(\Sigma_m) \hookrightarrow \Lambda \]
which sends a representation to its character. We can therefore extend the map $t$ to an endomorphism of $\Q \otimes \Lambda$ which stabilizes the lattices $\Lambda$ and $\rep(\Sigma_m)$ inside of it. It follows that the characteristic polynomial $\chi_t$ of $t$ can be computed mod $p$ with respect to either $\Lambda$ or $\rep(\Sigma_m)$, since both are of maximal rank in $\Q \otimes \Lambda$. In other words, it is enough to show that $\chi_t$ restricted to the lattice $\Lambda$ is $\chi_t(x) \equiv x^{p(m)} \mod p$, where $p(m)$ denotes the number of partitions of $m$. 

To this end, let $\delta_c$ be the characteristic function on the conjugacy class $c$ of $\Sigma_m$, determined by a specific cycle class $C_c$ with type $k=k(c)=(k_1,\ldots,k_l)$. Fix an ordered $p$-tuple $i_1,\ldots,i_p$ with $i_1+\cdots+i_p=m$, let $H=\Sigma_{i_1} \times \cdots \times \Sigma_{i_p}$ denote the corresponding subgroup of $\Sigma_m$, and consider $\mathrm{Ind}_H^{\Sigma_m} \circ \mathrm{Res}_H^{\Sigma_m}$. In general, the formula for inducing a class function $f$ from a subgroup $K \subset G$ is $\mathrm{Ind}_{K}^{G}(f)(g) = \sum_{s \in K \backslash G} f_0(s^{-1}gs)$ where is $f_0(x) = f(x)$ for $x \in K$ and $0$ otherwise. Specializing to our situation we see that
\[ \mathrm{Ind}_H^{\Sigma_m} \circ \mathrm{Res}_H^{\Sigma_m}(\delta_c) = m_c(H) \cdot\delta_c \]
where $m_c(H)$ is some non-negative integer. It follows from this description that 
\[ \sum_{i_1+\cdots+i_p=m} \mathrm{Ind}^{\Sigma_m}_{\Sigma_{i_1} \times \cdots \times \Sigma_{i_p}} \circ \mathrm{Res}^{\Sigma_m}_{\Sigma_{i_1} \times \cdots \times \Sigma_{i_p}}(\delta_c) = \sum_{i_1+\cdots+i_p=m} m_c(\Sigma_{i_1} \times \cdots \times \Sigma_{i_p}) \cdot\delta_c \]
so the matrix of $t$ on $\Lambda$ with respect to the basis given by characteristic functions on the conjugacy classes is diagonal with eigenvalues $m_c = \sum_{i_1+\cdots+i_p=m} m_c(\Sigma_{i_1} \times \cdots \times \Sigma_{i_p})$. We claim that, if the cycle type of $c$ is $k=(k_1,\ldots,k_l)$, then $m_c = p^l$. The proposition follows immediately.

So let us prove this claim. In fact, we will identify $m_c$ with the number of functions from a set $K_c$ with $l$ elements to $\{1,\ldots,p \}$. To this end, note that it is enough to evaluate the above formula on a fixed element $g \in c= c_g$. If $H= \Sigma_{i_1} \times \cdots \times \Sigma_{i_p}$, then any element $s \in H \backslash \Sigma_m$ has $\mathrm{Res}_H^{\Sigma_m}(\delta_c)(s^{-1}gs) \ne 0$ if and only if it permutes the cycles in the cycle representation of $c_g$ such that they are compatible with the cycle structure of $H= \Sigma_{i_1} \times \cdots \times \Sigma_{i_p}$. Therefore, we can label the cycles in the cycle representation of $c_g$ by the numbers $j \in \{1,\ldots,p\}$ corresponding to which factor $\Sigma_{i_j}$ of $H$ they belong to after conjugation by $s$. This induces a bijection between pairs $(H,s \in H \backslash \Sigma_m)$ that contribute to $\sum_{i_1+\cdots+i_p=m} m_c(\Sigma_{i_1} \times \cdots \times \Sigma_{i_p})$, i.e., such that 
\[\mathrm{Res}_H^{\Sigma_m}(\delta_c)(s^{-1}gs) \ne 0,\]
and all possible labelings of the $l$ cycles in $c_g$ by numbers $\{1,\ldots,p\}$, i.e., functions from $K_c$ to $\{1,\ldots,p \}$: Surjectivity is clear, while injectivity uses the fact that the ordered tuple $(i_1,\ldots,i_p)$ is determined by the labeling, and the coset representative $s$ can be recovered from the given permutation of the cycles in the cycle representation of $c$, up to conjugation in $H=\Sigma_{i_1} \times \cdots \times \Sigma_{i_p}$. This completes the proof.
\end{proof}

\begin{rmk}
In fact, the argument shows more, namely that the eigenvalues of $t=t(m,p)$, thought of as an endomorphism of $\rep(\Sigma_m)$ or $\Lambda$, are precisely all the $p^{l_c}$ where $c$ runs through the conjugacy classes of $\Sigma_m$ and $l_c$ denotes the length of the cycle representation of $c$. 
\end{rmk}

\begin{proof}[Alternative proof of \myref{mainreptheory}]
We sketch another argument, due to Charles Rezk, which is closer to the perspective in Section \ref{height1}. Define $A_m=\Hom(\rep(\Sigma_m),\Z)$, then 
\[ \bigoplus_m A_m = \Lambda [x],\]
the free $\lambda$-ring on one generator $x$ of degree 1; note that the map 
\[t(p)=\bigoplus_m t(m,p)\colon  \Lambda[x] \to \Lambda[x]\]  
is given by $x \mapsto px$. Let $x_k = \psi_k(x)$, where $\psi_k$'s are the Adams operations, and consider the subring $A = \Z[x_1,x_2,\ldots]\subset \Lambda[x]$ with $x_k$ of degree $k$. Since $A \otimes \Q \cong \Lambda[x] \otimes \Q$, we can compute the eigenvalues of $t(p)$ in terms of the monomial basis given by the $x_k$:
\[\begin{array} {lcl} 
t(p)(x_{k_1}\cdots x_{k_n}) & = & \psi_{k_1}(px)\cdots \psi_{k_n}(px) \\ 
& = & p^n \psi_{k_1}(x)\cdots \psi_{k_n}(x) \\
& = & p^nx_{k_1}\cdots x_{k_n}.  
\end{array}\]
So the monomial basis in the $x_k$'s provides a basis of eigenvectors for $t(p)$, with the eigenvalues as above. 
\end{proof}

\bibliographystyle{alpha}
\bibliography{CompletedPowerOps_Bibliography}

\begin{thebibliography}{EKMM97}

\bibitem[AH94]{formal}
Matthew Ando and Michael~J. Hopkins.
\newblock Notes on formal groups.
\newblock Unpublished notes based on a course taught by Michael Hopkins at MIT
  in the spring of 1990., 1994.

\bibitem[And95]{andothesis}
Matthew Ando.
\newblock Isogenies of formal group laws and power operations in the cohomology
  theories {$E_n$}.
\newblock {\em Duke Math. J.}, 79(2):423--485, 1995.

\bibitem[ARV10]{siftedcolimits}
J.~Ad{\'a}mek, J.~Rosick{\'y}, and E.~M. Vitale.
\newblock What are sifted colimits?
\newblock {\em Theory Appl. Categ.}, 23:No. 13, 251--260, 2010.

\bibitem[ARV11]{algtheories}
J.~Ad{\'a}mek, J.~Rosick{\'y}, and E.~M. Vitale.
\newblock {\em Algebraic theories}, volume 184 of {\em Cambridge Tracts in
  Mathematics}.
\newblock Cambridge University Press, Cambridge, 2011.
\newblock A categorical introduction to general algebra, With a foreword by F.
  W. Lawvere.

\bibitem[Bak09]{completehopfalgebroids}
Andrew Baker.
\newblock {$L$}-complete {H}opf algebroids and their comodules.
\newblock In {\em Alpine perspectives on algebraic topology}, volume 504 of
  {\em Contemp. Math.}, pages 1--22. Amer. Math. Soc., Providence, RI, 2009.

\bibitem[BK72]{holim}
A.~K. Bousfield and D.~M. Kan.
\newblock {\em Homotopy limits, completions and localizations}.
\newblock Lecture Notes in Mathematics, Vol. 304. Springer-Verlag, Berlin,
  1972.

\bibitem[BMMS86]{hinfty}
R.~R. Bruner, J.~P. May, J.~E. McClure, and M.~Steinberger.
\newblock {\em {$H_\infty $} ring spectra and their applications}, volume 1176
  of {\em Lecture Notes in Mathematics}.
\newblock Springer-Verlag, Berlin, 1986.

\bibitem[Bou96]{bousfield}
A.~K. Bousfield.
\newblock On {$\lambda$}-rings and the {$K$}-theory of infinite loop spaces.
\newblock {\em $K$-Theory}, 10(1):1--30, 1996.

\bibitem[BR12]{bktaq}
Mark Behrens and Charles Rezk.
\newblock The {B}ousfield-{K}uhn functor and topological {A}ndr\'e-{Q}uillen
  cohomology.
\newblock Preprint, 2012.

\bibitem[BW05]{pleth}
James Borger and Ben Wieland.
\newblock Plethystic algebra.
\newblock {\em Adv. Math.}, 194(2):246--283, 2005.

\bibitem[Dem86]{demazure}
Michel Demazure.
\newblock {\em Lectures on {$p$}-divisible groups}, volume 302 of {\em Lecture
  Notes in Mathematics}.
\newblock Springer-Verlag, Berlin, 1986.
\newblock Reprint of the 1972 original.

\bibitem[EKMM97]{ekmm}
A.~D. Elmendorf, I.~Kriz, M.~A. Mandell, and J.~P. May.
\newblock {\em Rings, modules, and algebras in stable homotopy theory},
  volume~47 of {\em Mathematical Surveys and Monographs}.
\newblock American Mathematical Society, Providence, RI, 1997.
\newblock With an appendix by M. Cole.

\bibitem[GH03]{ghproblems}
P.~G. Goerss and M.~J. Hopkins.
\newblock Moduli problems for structured ring spectra.
\newblock Preprint, 2003.

\bibitem[GH04]{ghspaces}
P.~G. Goerss and M.~J. Hopkins.
\newblock Moduli spaces of commutative ring spectra.
\newblock In {\em Structured ring spectra}, volume 315 of {\em London Math.
  Soc. Lecture Note Ser.}, pages 151--200. Cambridge Univ. Press, Cambridge,
  2004.

\bibitem[GHMR05]{ghmr}
P.~Goerss, H.-W. Henn, M.~Mahowald, and C.~Rezk.
\newblock A resolution of the {$K(2)$}-local sphere at the prime 3.
\newblock {\em Ann. of Math. (2)}, 162(2):777--822, 2005.

\bibitem[GM92]{gmcomplete}
J.~P.~C. Greenlees and J.~P. May.
\newblock Derived functors of {$I$}-adic completion and local homology.
\newblock {\em J. Algebra}, 149(2):438--453, 1992.

\bibitem[HMS94]{picard}
Michael~J. Hopkins, Mark Mahowald, and Hal Sadofsky.
\newblock Constructions of elements in {P}icard groups.
\newblock In {\em Topology and representation theory ({E}vanston, {IL}, 1992)},
  volume 158 of {\em Contemp. Math.}, pages 89--126. Amer. Math. Soc.,
  Providence, RI, 1994.

\bibitem[Hop14]{hopkinslocal}
Michael~J. Hopkins.
\newblock {$K(1)$}-local {$E_\infty$}-ring spectra.
\newblock In {\em Topological modular forms}, volume 201 of {\em Math. Surveys
  Monogr.}, pages 287--302. Amer. Math. Soc., Providence, RI, 2014.

\bibitem[Hov04]{hoveynotes}
Mark Hovey.
\newblock Some spectral sequences in {M}orava {$E$}-theory.
\newblock {}, 2004.

\bibitem[Hov08]{filteredcolim}
Mark Hovey.
\newblock Morava {$E$}-theory of filtered colimits.
\newblock {\em Trans. Amer. Math. Soc.}, 360(1):369--382 (electronic), 2008.

\bibitem[HS98]{nilp2}
Michael~J. Hopkins and Jeffrey~H. Smith.
\newblock Nilpotence and stable homotopy theory. {II}.
\newblock {\em Ann. of Math. (2)}, 148(1):1--49, 1998.

\bibitem[HS99]{moravak}
Mark Hovey and Neil~P. Strickland.
\newblock Morava {$K$}-theories and localisation.
\newblock {\em Mem. Amer. Math. Soc.}, 139(666):viii+100, 1999.

\bibitem[Joy02]{joyal}
A.~Joyal.
\newblock Quasi-categories and {K}an complexes.
\newblock {\em J. Pure Appl. Algebra}, 175(1-3):207--222, 2002.
\newblock Special volume celebrating the 70th birthday of Professor Max Kelly.

\bibitem[Lau03]{spin}
Gerd Laures.
\newblock An {$E_\infty$} splitting of spin bordism.
\newblock {\em Amer. J. Math.}, 125(5):977--1027, 2003.

\bibitem[Lur09]{htt}
Jacob Lurie.
\newblock {\em Higher topos theory}, volume 170 of {\em Annals of Mathematics
  Studies}.
\newblock Princeton University Press, Princeton, NJ, 2009.

\bibitem[Lur14]{ha}
Jacob Lurie.
\newblock Higher algebra.
\newblock Preprint, 2014.

\bibitem[ML98]{working}
Saunders Mac~Lane.
\newblock {\em Categories for the working mathematician}, volume~5 of {\em
  Graduate Texts in Mathematics}.
\newblock Springer-Verlag, New York, second edition, 1998.

\bibitem[Rez08]{ht2p2}
Charles Rezk.
\newblock Power operations for {M}orava {$E$}-theory of height 2 at the prime
  2.
\newblock Preprint, 2008.

\bibitem[Rez09]{congruence}
Charles Rezk.
\newblock The congruence criterion for power operations in {M}orava
  {$E$}-theory.
\newblock {\em Homology, Homotopy Appl.}, 11(2):327--379, 2009.

\bibitem[Rez12]{koszul}
Charles Rezk.
\newblock Rings of power operations for {M}orava {$E$}-theories are {K}oszul.
\newblock Preprint, 2012.

\bibitem[Rez13a]{analyticcompletion}
Charles Rezk.
\newblock Analytic completion.
\newblock Preprint, 2013.

\bibitem[Rez13b]{calculations}
Charles Rezk.
\newblock Power operations in morava e-theory: structure and calculations.
\newblock Preprint, 2013.

\bibitem[Ric02]{ahsstaq}
Birgit Richter.
\newblock An {A}tiyah-{H}irzebruch spectral sequence for topological
  {A}ndr\'e-{Q}uillen homology.
\newblock {\em J. Pure Appl. Algebra}, 171(1):59--66, 2002.

\bibitem[Rie14]{riehl}
Emily Riehl.
\newblock {\em Categorical homotopy theory}, volume~24 of {\em New Mathematical
  Monographs}.
\newblock Cambridge University Press, Cambridge, 2014.

\bibitem[Sal10]{approxsubcats}
Andrew Salch.
\newblock Approximation of subcategories by abelian subcategories.
\newblock Preprint, 2010.

\bibitem[{Sta}14]{stacksproject}
{Stacks Project Authors}.
\newblock Stacks project.
\newblock \url{http://stacks.math.columbia.edu}, 2014.

\bibitem[Str98]{symmgps}
N.~P. Strickland.
\newblock Morava {$E$}-theory of symmetric groups.
\newblock {\em Topology}, 37(4):757--779, 1998.

\bibitem[Str00]{stricklandduality}
N.~P. Strickland.
\newblock Gross-{H}opkins duality.
\newblock {\em Topology}, 39(5):1021--1033, 2000.

\bibitem[Val15]{val}
Gabriel Valenzuela.
\newblock {\em Homological algebra of complete and torsion modules}.
\newblock 2015.
\newblock Thesis (Ph.D.)--Wesleyan University.

\bibitem[Yau10]{lambda}
Donald Yau.
\newblock {\em Lambda-rings}.
\newblock World Scientific Publishing Co. Pte. Ltd., Hackensack, NJ, 2010.

\bibitem[Yek11]{noetherian}
Amnon Yekutieli.
\newblock On flatness and completion for infinitely generated modules over
  {N}oetherian rings.
\newblock {\em Comm. Algebra}, 39(11):4221--4245, 2011.

\bibitem[Zhu14]{ht2p3}
Yifei Zhu.
\newblock The power operation structure on {M}orava {$E$}-theory of height 2 at
  the prime 3.
\newblock {\em Algebr. Geom. Topol.}, 14(2):953--977, 2014.

\end{thebibliography}

\end{document}